\documentclass[10pt]{article}
\RequirePackage{amsthm,amsmath,amsfonts,amssymb}
\RequirePackage[numbers]{natbib}
\usepackage{lipsum}
\usepackage{graphicx}
\usepackage{epstopdf}
\usepackage{algorithmic}
\usepackage{mathtools}
\usepackage{mathrsfs}

\usepackage[a4paper,  margin=1in]{geometry}

\DeclareMathOperator*{\argmax}{arg\,max}
\DeclareMathOperator*{\argmin}{arg\,min}

\theoremstyle{plain}
\newtheorem{theorem}{Theorem}
\newtheorem{corollary}{Corollary}
\newtheorem{lemma}{Lemma}
\newtheorem{proposition}{Proposition}
\theoremstyle{remark}
\newtheorem{condition}{Condition}

\title{The infinite Viterbi alignment and decay-convexity}

\author{Nick Whiteley\thanks{%
    School of Mathematics, University of Bristol and the Alan Turing Institute.}, 
  Matt W. Jones \thanks{%
	School of Physiology, Pharmacology and Neuroscience, University of Bristol.}, 
  Aleks Domanski\thanks{School of Physiology, Pharmacology and Neuroscience, University of Bristol and the Alan Turing Institute.}           
  }

\begin{document}

\maketitle
\begin{abstract}
The infinite Viterbi alignment is the limiting maximum a-posteriori estimate
of the unobserved path in a hidden Markov model as the length of the
time horizon grows. For models on state-space
$\mathbb{R}^{d}$ satisfying a new ``decay-convexity'' condition,
we develop an approach to existence of the infinite Viterbi alignment in an infinite dimensional
Hilbert space. Quantitative bounds on the distance to the Viterbi
process, which are the first of their kind, are derived and used to illustrate how approximate estimation via parallelization can 
be accurate and scaleable to high-dimensional problems because the
rate of convergence to the infinite Viterbi alignment does not necessarily depend
on $d$. The results are applied to approximate estimation via parallelization and a model of neural population activity.
\end{abstract}






\section{Introduction}\label{subsec:Background-and-motivation}


Let  the \emph{signal} $(X_{n})_{n\geq 0}$ 
be a Markov chain with state space $\mathbb{R}^{d}$ whose initial
distribution and transition kernel admit densities $\mu(x)$ and $f(x,x^{\prime})$
with respect to $d$-dimensional Lebesgue measure. Let $(Y_{n})_{n\geq 0}$, called the \emph{observations},
be each valued in a measurable space $(\mathbb{Y},\mathcal{Y})$, 
conditionally independent given $(X_{n})_{n\geq 0}$ and
such that for any $A\in\mathcal{Y}$, the conditional probability
of $\{Y_{n}\in A\}$ given $(X_{n})_{n\geq 0}$ can be
written in the form $\int_{A}g(X_{n},y)\rho(\mathrm{d}y)$, where
$g:\mathbb{R}^{d}\times\mathbb{Y}\to[0,+\infty)$ and $\rho$ is a
measure on $\mathcal{Y}$. Models of this form, under the names hidden Markov models (HMMs) or state space models, are applied in a huge variety of fields including econometrics,
engineering, ecology, machine learning and neuroscience \cite{west2006bayesian,durbin2012time,douc2014nonlinear}. With:
\begin{align}
U_n(x_{0},\ldots,x_{n},y_{0},\ldots,y_{n} )&\coloneqq-\log\mu(x_{0})-\log g(x_{0},y_{0})\nonumber\\
&-\sum_{m=1}^{n}\log f(x_{m-1},x_{m})-\sum_{m=1}^{n}\log g(x_{m},y_{m}),\label{eq:U^m_intro}
\end{align}
the maximum
a-posteriori path estimation problem given $(y_{0},\ldots,y_{n})$
is to find:
\begin{equation}
(\xi_{n,0},\ldots,\xi_{n,n})\coloneqq \argmin_{x_0,\ldots,x_n}\;U_n(x_{0},\ldots,x_{n},y_{0},\ldots,y_{n}).\label{eq:optimization_problem}
\end{equation}
The dependence of $\xi_{n,0},\ldots,\xi_{n,n}$ on $y_{0},\ldots,y_{n}$ is not shown in the notation. With $y\coloneqq(y_0,y_1,\ldots)\in\mathbb{Y}^{\mathbb{N}}$, the \emph{infinite Viterbi alignment} is a sequence $\xi_{\infty}=(\xi_{\infty,n})_{n\geq 0}$
such that for any $m\in\mathbb{N}_{0}$,
\begin{equation}
(\xi_{\infty,0},\ldots,\xi_{\infty,m})=\lim_{n\to\infty}(\xi_{n,0},\ldots,\xi_{n,m}),\quad\text{for }\quad \mathbf{P}_Y \text{-almost all } y,\label{eq:Viterbi}
\end{equation}
where in general  $\xi_{\infty}=(\xi_{\infty,n})_{n\geq 0}$ is an unknown function of the infinite observation sequence $y=(y_0,y_1,\ldots)$ and $\mathbf{P}_Y$ is some probability measure on $\mathcal{Y}^{\otimes\mathbb{N}}$. In the ``well-specified'' case,  where $\mathbf{P}_Y$ is the probability law of the observation process $Y_0,Y_1,\ldots$ induced by the same $\mu$, $f$ and $g$ defining each function $U_n$ in \eqref{eq:U^m_intro},  the infinite Viterbi alignment as a function of $Y_0,Y_1,\ldots$ is called the \emph{Viterbi process}. 

Whilst the existence and uniqueness of the left hand side of \eqref{eq:optimization_problem} can be addressed directly in terms of $\mu$, $f$, and $g$, the existence and uniqueness of the infinite Viterbi alignment and Viterbi process are less obvious. Studies and applications   to date have mostly focused on the case where the state space of the signal is a discrete finite set
\cite{caliebe2002convergence,caliebe2006properties,koloydenko2008infinite,lember2008adjusted,lember2010constructive,kuljus2012asymptotic} and there 
 the convergence in (\ref{eq:Viterbi})
is with respect to the discrete metric. In this discrete setting the infinite Viterbi alignment and Viterbi process have been also been studied recently for pairwise Markov models \citep{lember2020existence,lember2021regenerativity}. However, the case of HMMs with signal state-space $\mathbb{R}^{d}$ is considerably less well-understood: the only 
work known to the authors which considers this case is \cite{chigansky2011viterbi}, where convergence in (\ref{eq:Viterbi}) is with respect to Euclidean distance, and proofs are given only for the case $d=1$. The work \cite{chigansky2011viterbi} considered two approaches to existence of the Viterbi process, one based on regeneration times which is reminiscent of ideas used in the discrete case, and one involving dynamic programming operators. 

The assumptions in the present work are most directly comparable to those used in the dynamic programming approach of \cite{chigansky2011viterbi}---we shall discuss this in more detail in section \ref{subsec:Comparison-to-the}---but our setup and proof techniques are quite different: we develop a new framework in which the infinite Viterbi alignment is an element of an infinite dimensional Hilbert space, $l_{2}(\gamma)$,
where $\gamma\in(0,1]$ is a parameter related to the rate of convergence to the infinite Viterbi alignment. This approach has several benefits. Firstly, it allows
interpretable quantitative bounds to be obtained which measure the distance to the infinite Viterbi alignment in a
norm on $l_{2}(\gamma)$ which gives a stronger notion of convergence
than the pointwise convergence in (\ref{eq:Viterbi}).  Secondly, via
a new ``decay-convexity'' condition our approach provides a new characterization of the Viterbi
process as the fixed point of an infinite dimensional ordinary differential
equation which arises in the limit $n\to\infty$. Thirdly, our analysis conveniently allows us to handle misspecified hidden Markov models: the situation in which the probability measure $\mathbf{P}_Y$ in \eqref{eq:Viterbi} need not be the probability law of $Y_0,Y_1,\ldots$ under the HMM specified by $\mu$, $f$, $g$, nor necessarily of the law of the observation process under any HMM. Misspecification is a topic of interest in the  study of asymptotic properties of maximum likelihood estimators \citep{douc2012asymptotic} and forgetting of initial conditions \citep{Douc2009}[Remark 2] for HMMs.  It is important in the present context because in practice there is usually some modelling error, and one computes maximum a-posteriori estimator as in \eqref{eq:optimization_problem} with little if any knowledge of the probability law of the observation process.

A shortcoming of the present work is that we do not investigate ergodic or regenerative properties of the Viterbi process, as conducted in the discrete case in, for example, \cite{caliebe2006properties}. We leave this as a topic for future research. Another restriction of our approach is that we consider only situations in which the infinite Viterbi alignment, when it exists, is unique, reflecting the fact that our conditions are related to convexity of $U_n(\cdot,y)$. The question of whether our approach can be extended to situations in which there is non-uniqueness is left for future research.

Existence of the infinite Viterbi alignment can  be understood as meaning that maximum a-posteriori state estimates have vanishing dependence on the observations in the distant future. From a practical point of view, this is important because it indicates that such future observations may be largely irrelevant to estimation at the present time and so can be safely ignored, resulting in computational savings. This idea can be extended to design a scheme for parallelized approximate computation of $(\xi_{n,0},\ldots,\xi_{n,n})$ to which we shall apply our quantitative convergence results.  

We start with definitions and Condition \ref{assu:hmm} concerning $\mu$, $f$, $g$ in section \ref{sec:defns_and_assumptions}, leading to the statements of Theorems \ref{thm:main_non_uniform} and \ref{thm:main_uniform}, which establish quantitative bounds on the differences between  maximum
a-posteriori estimators as a function of the observation sequence $y=(y_0,y_1,\ldots)$. Section \ref{sec:Discussion-and-application} maps out and demonstrates the two main steps involved in applying Theorem \ref{thm:main_non_uniform} or \ref{thm:main_uniform}, and the use of these theorems to establish the existence of the infinite Viterbi alignment. This is achieved via some intermediate results for classes of HMMs, combined with example calculations for specific HMMs to illustrate the reasoning involved.  Section \ref{sec:Discussion-and-application}  also contains discussion of how the  quantitative $l_{2}(\gamma)$  bounds can be used to analyze the error associated with the aforementioned parallelized approximate estimation scheme, and an application with numerical results for a model of neural population activity.

\section{Definitions}\label{sec:defns_and_assumptions}

With $d\in\mathbb{N}$ fixed throughout, when $x$ is a point in $\mathbb{R}^{\mathbb{N}}$, we associate with it the vectors $x_{0},x_{1},\ldots$,
each in $\mathbb{R}^{d}$, such that $x=[x_{0}^{\mathrm{T}}\,x_{1}^{\mathrm{T}}\,\ldots]^{\mathrm{T}}$. With $\left\langle \cdot,\cdot\right\rangle $ and $\|\cdot\|$ the
Euclidean inner product and norm on $\mathbb{R}^{d}$, define
the inner product and norm on $\mathbb{R}^{\mathbb{N}}$ associated
with a given $\gamma\in(0,1]$,
\[
\left\langle x,x^{\prime}\right\rangle _{\gamma}\coloneqq\sum_{n=0}^{\infty}\gamma^{n}\left\langle x_{n},x_{n}^{\prime}\right\rangle ,\qquad\|x\|_{\gamma}=\left\langle x,x\right\rangle _{\gamma}^{1/2}\coloneqq\left(\sum_{n=0}^{\infty}\gamma^{n}\|x_{n}\|^{2}\right)^{1/2}.
\]
Let $l_{2}(\gamma)$ be the Hilbert space consisting of the set $\{x\in\mathbb{R}^{\mathbb{N}}:\|x\|_{\gamma}<\infty\}$
equipped with the inner-product $\left\langle \cdot,\cdot\right\rangle _{\gamma}$
and the usual element-wise addition and scalar multiplication of vectors
over field $\mathbb{R}$. For each $n\geq 0$, $l_{2}^{n}(\gamma)$
denotes the subspace consisting of those $x\in l_{2}(\gamma)$ such
that $x_{m}=0$ for $m>n$, with the convention that $l_{2}^{\infty}(\gamma)\equiv l_{2}(\gamma)$.
Note that for $n<\infty$ the set of vectors $l_{2}^{n}(\gamma)$ does not actually depend on $\gamma$;  the notation $l_{2}^{n}(\gamma)$ is used as a reminder that $l_{2}^{n}(\gamma)$ is a subspace of $l_{2}(\gamma)$. Let $\|x\|_{\gamma,n}\coloneqq \left(\sum_{m=0}^{\infty}\gamma^{|m-n|}\|x_{m}\|^{2}\right)^{1/2}$. When $y\in\mathbb{Y}^{\mathbb{N}}$ we shall identify  $y=(y_{0},y_{1},\ldots)$ where each $y_n\in\mathbb{Y}$,  stopping short of explicitly regarding $y$ as a vector since for our main results we shall not need to assume that $\mathbb{Y}$ or $\mathbb{Y}^\mathbb{N}$ is a vector space.

For $x\in\mathbb{R}^{\mathbb{N}}$ and $y\in\mathbb{Y}^{\mathbb{N}}$ define
\begin{align}
\phi_n(x,y) & \coloneqq \log f(x_{n-1},x_{n})+\log f(x_{n},x_{n+1})+\log g(x_{n},y_{n}),\quad n\geq1,\label{eq:phi_n_defn}\\
\tilde{\phi}_n(x,y) &  \coloneqq \begin{cases}
\log\mu(x_{0})+\log f(x_{0},x_{1})+\log g(x_{0},y_{0}), & n=0,\\
\log f(x_{n-1},x_{n})+\log g(x_{n},y_{n}),\qquad & n\geq1.
\end{cases}\label{eq:phi_tilde_n_defn}
\end{align}
and for each $n\geq0$ let $\nabla_{n}\phi_n(x,y)$ and $\nabla_{n}\tilde{\phi}_n(x,y)$
be the vectors in $\mathbb{R}^{d}$ whose $i$th entries are the partial
derivatives of $\phi_n(x,y)$ and $\tilde{\phi}_n(x,y)$ with respect
to the $i$th entry of $x_{n}$ (the existence of such derivatives
is part of Condition \ref{assu:hmm} below).

For each $n\geq 0$, define the vector field $ \nabla U_n:\mathbb{R}^{\mathbb{N}}\times\mathbb{Y}^{\mathbb{N}}\to \mathbb{R}^{\mathbb{N}}$,
\begin{multline}
\nabla U_n(x,y)\coloneqq \\-[\nabla_{0}\tilde{\phi}_{0}(x,y)^{\mathrm{T}},\;\nabla_{1}\phi_{1}(x,y)^{\mathrm{T}},\;\cdots\;\nabla_{n-1}\phi_{n-1}(x,y)^{\mathrm{T}},\;\nabla_{n}\tilde{\phi}_n(x,y)^{\mathrm{T}},\;0,\;0,\;\cdots]^{\mathrm{T}}.\label{eq:gradU^N_defn}
\end{multline}
With these definitions, the first $d(n+1)$ elements of the vector
$\nabla U_n(x,y)$ are the partial derivatives of $U_n(x_{0},\ldots,x_{n},y_{0},\ldots,y_{n})$
with respect to $(x_{0},\ldots, x_{n})$,
whilst the other elements of the vector $\nabla U_n(x,y)$ are zero.
Treatment of  $\nabla U_n(x,y)$ as an infinitely long
vector allows us to consider
$(\nabla U_n(\cdot,y))_{n\geq 0}$ as a sequence of vector fields on the common Hilbert space
$l_{2}(\gamma)$. Define also
\begin{align}
\alpha_{\gamma,n}(y)\coloneqq\sum_{m=0}^{n}\gamma^{n-m}\beta_m(y), & \quad\beta_m(y)\coloneqq\|\nabla_{m}\phi_{m}(0,y)\|^{2}\vee\|\nabla_{m}\tilde{\phi}_{m}(0,y)\|^{2},\label{eq:alpha_and_beta_defns}
\end{align}
\begin{equation}
\eta_{n}(r,y)\coloneqq \sup_{\|x\|_{\gamma,n}^{2}\leq r}\|\nabla_{n}\phi_n(x,y)\|^{2}\vee\|\nabla_{n}\tilde{\phi}_n(x,y)\|^{2}.\label{eq:eta_defn}
\end{equation}

\section{Quantitative $l_{2}(\gamma)$ bounds for maximum a-posteriori estimators}\label{sec:quantitative_bounds}

\begin{condition}
\label{assu:hmm}$\,$

 a) $\mu$, $f$, and $g(\cdot,y)$ for all $y\in\mathbb{Y}$,
are everywhere strictly positive and continuously differentiable.

 b) there exist constants $\zeta,\tilde{\zeta},\theta$ such
that $0\leq\theta<\zeta/2 \wedge \tilde{\zeta}$, and for all $x,x^{\prime}\in\mathbb{R}^{\mathbb{N}}$ and all $y\in\mathbb{Y}^{\mathbb{N}}$,
\begin{align*}
 & \left\langle x_{n}-x_{n}^{\prime},\nabla_{n}\phi_n(x,y)-\nabla_{n}\phi_{n}(x^{\prime},y)\right\rangle \\
 & \leq-\zeta\|x_{n}-x_{n}^{\prime}\|^{2}+\theta\|x_{n}-x_{n}^{\prime}\|\left(\|x_{n-1}-x_{n-1}^{\prime}\|+\|x_{n+1}-x_{n+1}^{\prime}\|\right),\quad\forall n\geq1,\\
 & \left\langle x_{n}-x_{n}^{\prime},\nabla_{n}\tilde{\phi}_n(x,y)-\nabla_{n}\tilde{\phi}_{n}(x^{\prime},y)\right\rangle \\
 & \leq\begin{cases}
-\tilde{\zeta}\|x_{0}-x_{0}^{\prime}\|^{2}+\theta\|x_{0}-x_{0}^{\prime}\|\|x_{1}-x_{1}^{\prime}\|, & n=0,\\
-\tilde{\zeta}\|x_{n}-x_{n}^{\prime}\|^{2}+\theta\|x_{n}-x_{n}^{\prime}\|\|x_{n-1}^{\prime}-x_{n-1}^{\prime}\|, & \forall n\geq1.
\end{cases}
\end{align*}
\end{condition}

\begin{theorem}
\label{thm:main_non_uniform}Assume that Condition \ref{assu:hmm}
holds, and with $\zeta,\tilde{\zeta},\theta$ as therein, let $\gamma$
be any value in $(0,1]$ such that:
\begin{equation}
\zeta>\theta\frac{(1+\gamma)^{2}}{2\gamma},\qquad\tilde{\zeta}>\theta\frac{\left(1+\gamma\right)}{2\gamma}.\label{eq:theta_gamma_inequality}
\end{equation}
Then with any $\lambda$ such that:
\begin{equation}
0<\lambda\leq\left\{ \zeta-\theta\frac{(1+\gamma)^{2}}{2\gamma}\right\} \wedge\left\{ \tilde{\zeta}-\theta\frac{\left(1+\gamma\right)}{2\gamma}\right\} ,\label{eq:lambda_conditions}
\end{equation}
and any $n\geq 0$,
\begin{equation}
\left\langle x-x^{\prime},\nabla U_n(x,y)-\nabla U_n(x^{\prime},y)\right\rangle _{\gamma}\geq\lambda\|x-x^{\prime}\|_{\gamma}^{2},\quad \forall x,x^{\prime}\in l_{2}^{n}(\gamma),y\in\mathbb{Y}^{\mathbb{N}}.\label{eq:U^n_decay_convexity}
\end{equation}
For any $y\in\mathbb{Y}^{\mathbb{N}}$, amongst all the vectors in $l_{2}^{n}(\gamma)$ there is a unique
vector $\xi_{n}$ such that $\nabla U_n(\xi_{n},y)=0$, and
\begin{multline}
\sup_{m\in\mathbb{N}_{0},m\geq n}\|\xi_{n}-\xi_{m}\|_{\gamma}^{2}\leq \frac{ \gamma^{n}}{\lambda^{2}}\eta_{n}\left(\lambda^{-2}\alpha_{\gamma,n}(y),y\right)\\+\frac{\gamma^{n+1}}{\lambda^{2}}\eta_{n+1}\left(\gamma\lambda^{-2}\alpha_{\gamma,n}(y),y\right)+\frac{1}{\lambda^{2}}\sum_{k=n+2}^{\infty}\gamma^{k}\beta_{k}(y).\label{eq:thm_1_bound}
\end{multline}
\end{theorem}
\noindent The proof of Theorem \ref{thm:main_non_uniform} is in appendix
\ref{subsec:Proofs-of-the}. We pass several remarks on this result.
\begin{itemize}
\item To connect Theorem \ref{thm:main_non_uniform}  with the setup from section \ref{subsec:Background-and-motivation}, note that since $\nabla U_n(\xi_{n},y)=0$, the first $d(n+1)$ elements of
the vector $\xi_{n}$ solve the estimation problem \eqref{eq:optimization_problem},
and since $\xi_{n}\in l_{2}^{n}(\gamma)$, the remaining elements
of $\xi_{n}$ are zero. Note that the dependence of $\xi_n$ on $y$ is not shown in the notation. 
\item When Condition \ref{assu:hmm} holds, there always exists $\gamma\in(0,1)$
satisfying (\ref{eq:theta_gamma_inequality}) and $\lambda$ satisfying
(\ref{eq:lambda_conditions}) because $\lim_{\gamma\to 1} (1+\gamma)^{2}/2\gamma=\lim_{\gamma\to 1} (1+\gamma)/\gamma = 2$
and Condition \ref{assu:hmm} requires $0\leq\theta<\zeta/2 \wedge\tilde{\zeta}$.
\item If the right hand
side of (\ref{eq:thm_1_bound}) converges to zero as $n\to\infty$,
then $(\xi_{n})_{n\geq 0}$ is a Cauchy sequence in $l_{2}(\gamma)$. Whether or not this convergence occurs depends on value of $\gamma$, the ingredients $\mu$, $f$ and $g$ of the HMM and the $y$ in question, and we shall return this topic in section \ref{sec:Discussion-and-application}.
\end{itemize}
From hereon (\ref{eq:U^n_decay_convexity}) will be referred to as
``decay-convexity'' of $U_n$. Note that when $\gamma=1$, (\ref{eq:U^n_decay_convexity})
says exactly that $p(x_{0},\ldots,x_{n}|y_{0},\ldots,y_{n})$
is $\lambda$-strongly log-concave in $(x_{0},\ldots,x_{n})$ in the sense of \cite{saumard2014log}, which guarantees that $p(x_{0},\ldots,x_{n}|y_{0},\ldots,y_{n})$ has a unique maximiser. This begs the question, if $(\xi_{n,0},\cdots,\xi_{n,n})$ is a maximiser of $p(x_{0},\ldots,x_{n}|y_{0},\ldots,y_{n})$, then can $\xi_\infty$ be characterized as a maximiser of some function?  As we shall see next, the answer is to this question appears to be ``no'', but an alternative and closely related interpretation of $\xi_\infty$ is available if one introduces the vector field $ \partial U:\mathbb{R}^{\mathbb{N}}\times \mathbb{Y}^{\mathbb{N}}\to \mathbb{R}^{\mathbb{N}}$, 
\begin{equation}
\partial U(x,y)\coloneqq-[\nabla_{0}\tilde{\phi}_{0}(x,y)^{\mathrm{T}},\;\nabla_{1}\phi_{1}(x,y)^{\mathrm{T}},\;\nabla_{2}\phi_{2}(x,y)^{\mathrm{T}},\;\cdots]^{\mathrm{T}}.\label{eq:grad_U_infty_defn}
\end{equation}
\begin{condition}
\label{assu:hmm2}$\,$
For some given $y\in\mathbb{Y}^{\mathbb{N}}$,

\noindent a) there exists a finite constant $\chi$ such that for
all $n$ and $x\in l_{2}(\gamma)$,
\[
\|\nabla_{n}\phi_n(x,y)\|^{2}\vee\|\nabla_{n}\tilde{\phi}_n(x,y)\|^{2}\leq\beta_n(y)+\chi\left(\|x_{n-1}\|^{2}+\|x_{n}\|^{2}+\|x_{n+1}\|^{2}\right),
\]
b) $\sum_{n=0}^{\infty}\gamma^{n}\beta_n(y)<\infty$,

\noindent c) $x\mapsto \partial U(x,y)$ is continuous in $l_{2}(\gamma)$.
\end{condition}
Theorem \ref{thm:main_uniform} below shows that, under the assumptions of Theorem \ref{thm:main_non_uniform} and additionally Condition \ref{assu:hmm2}, $\xi_\infty$ is the fixed point of an infinite dimensional ODE, associated with the vector field $\partial U(\cdot,y)$. It is important to note here that 
the vector $\partial U(x,y)\in\mathbb{R}^{\mathbb{N}}$ is the element-wise limit as $n\to\infty$ of
the vector $\nabla U_n(x,y)$. Indeed it can be read off from (\ref{eq:gradU^N_defn})
that for any given $x$ and $y$, each element of the vector $\nabla U_n(x,y)$ is constant in
$n$ for all $n$ large enough.  However, in general for some given $x,y$, the limit $\lim_{n\to\infty}U_n(x,y)$ does not exist, so it is a characterization of $\xi_\infty$, as a fixed point of an ODE, rather than a maximiser of some function which is fruitful.


\begin{theorem}
\noindent \label{thm:main_uniform}In addition to the assumptions
of Theorem \ref{thm:main_non_uniform} and with $\gamma$ as therein,
assume that for some given $y\in\mathbb{Y}^{\mathbb{N}}$ Condition \ref{assu:hmm2} holds. Then with $\lambda$ as in Theorem \ref{thm:main_non_uniform},
\begin{equation}
\left\langle x-x^{\prime},\partial U(x,y)-\partial U(x^{\prime},y)\right\rangle _{\gamma}\geq\lambda\|x-x^{\prime}\|_{\gamma}^{2},\qquad\text{for all }x,x^{\prime}\in l_{2}(\gamma)\label{eq:mod_convex_inf}.
\end{equation}
There exists a globally defined and unique flow $\Phi:(t,x)\in\mathbb{R}_{+}\times l_{2}(\gamma)\mapsto\Phi(t,x)\in l_{2}(\gamma)$
which solves the Fr\'echet ordinary differential equation,
\begin{equation}
\frac{\mathrm{d}}{\mathrm{d}t}\Phi(t,x)=-\partial U(\Phi(t,x),y),\qquad \Phi(0,x)=x,\label{eq:ode_inf}
\end{equation}
this flow has a unique fixed point, $\xi_{\infty}\in l_{2}(\gamma)$, and this point satisfies $\partial U(\xi_{\infty},y)=0$. With $(\xi_{n})_{n\geq 0}$ as in Theorem \ref{thm:main_non_uniform},
\begin{equation}
\sup_{m\in\mathbb{N}_{0}\cup\{\infty\},m\geq n}\|\xi_{n}-\xi_{m}\|_{\gamma}^{2}\leq\frac{1}{\lambda^{2}}\left(\gamma^{n-1}\alpha_{\gamma,n}(y)\frac{2\chi}{\lambda^{2}}+\sum_{k=n}^{\infty}\gamma^{k}\beta_{k}(y)\right),\label{eq:thm_2_bound}
\end{equation}
and $\|\xi_n-\xi_\infty\|_\gamma\to 0 $ as $n\to\infty$.
\end{theorem}
\noindent The proof of Theorem \ref{thm:main_uniform} is in appendix
\ref{subsec:Proofs-of-the}. In summary, the assumptions a)-b) of Condition \ref{assu:hmm2}  ensure that
$\partial U(\cdot,y)$ maps $l_{2}(\gamma)$ to itself. Combined with
the continuity in assumption c) in Condition \ref{assu:hmm2} and (\ref{eq:mod_convex_inf}), this
allows an existence and uniqueness result of \cite{deimling2006ordinary}
for dissipative ordinary differential equations on Banach spaces to
be applied. It is
from here that the Fr\'echet derivative (\ref{eq:ode_inf}) arises.
Background information about Fr\'echet derivatives is given in appendix
\ref{subsec:Fr=0000E9chet-derivatives}.

\section{Discussion}\label{sec:Discussion-and-application}

\subsection{A roadmap for applying Theorems \ref{thm:main_non_uniform}
and \ref{thm:main_uniform}}

Our next objective is to set out and demonstrate mathematical tools which enable application of Theorems \ref{thm:main_non_uniform}
or \ref{thm:main_uniform}. For Theorem \ref{thm:main_non_uniform} there are two main tasks in such an application: firstly verifying Condition \ref{assu:hmm}  and secondly proving that the right-hand side of \eqref{eq:thm_1_bound}  converges to zero as $n\to\infty$. Addressing these two tasks separately allows us to accommodate naturally the case of misspecified HMMs, in which the observation sequence $y=(y_{n})_{n\geq 0}$ appearing in \eqref{eq:thm_1_bound} need not be a realization from the HMM specified by $\mu$, $f$, $g$, nor in fact any HMM. In sections \ref{subsec:verifying} and \ref{sec:behaviour_of_alpha_and_beta} we give intermediate and  generic results to help tackle these two main tasks for classes of HMMs, with calculations for a specific example in section \ref{sec:heavy_tailed}. 

In order to apply Theorem \ref{thm:main_uniform}, one must additionally verify Condition \ref{assu:hmm2}, and calculations for a specific example are given in \ref{sec:non_stationary}. In section \ref{sec:parallel} we address the use of Theorem \ref{thm:main_uniform} to quantify the error associated with parallelized approximation to MAP estimators, demonstrated with numerical results in the context of a model of neural population activity in section \ref{subsec:neural_model}.

\subsection{Verifying Condition \ref{assu:hmm}}\label{subsec:verifying}
Lemma \ref{lem:AR1} focuses on a class of models with linear and Gaussian signals together with conditional distributions of the observations given signals which are not necessarily log-concave as a function of the signal. By contrast, Lemma \ref{lem:non_gauss} addresses a class of nonlinear, non-Gaussian signal models, combined with conditional distributions of observations given signal which are strongly log-concave. Here the minimum and maximum eigenvalues of a real, symmetric matrix, say $B$, are
denoted $\rho_{\mathrm{min}}(B)$, $\rho_{\mathrm{max}}(B)$.
\begin{lemma}
\label{lem:AR1}Assume a) and b):

\noindent a) The signal satisfies the following vector autoregressive model:
\begin{equation}
X_{n}=AX_{n-1}+b+W_{n},\label{eq:linear_gauss_signal}
\end{equation}
where for $n\in\mathbb{N}$, $W_{n}\sim\mathcal{N}(0,\Sigma)$ is
independent of other random variables, $X_{0}\sim\mathcal{N}(b_{0},\Sigma_{0})$
, $\Sigma$ and $\Sigma_{0}$ are positive definite, $A$ is a $d\times d$
matrix and $b$ and $b_0$ are length-$d$ vectors.

\noindent b) For each $y\in\mathbb{Y}$, $g(\cdot,y)$
is strictly positive, continuously differentiable and there exists
$\lambda_{g}\in\mathbb{R}$ such that for all $x,x^\prime\in\mathbb{R}^d$, $y\in\mathbb{Y}$,
\begin{equation}
\left\langle x-x^{\prime},\nabla_{x}\log g(x,y)-\nabla_{x}\log g(x^{\prime},y)\right\rangle \leq\lambda_{g}\|x-x^{\prime}\|^{2}.\label{eq:semi_log_concavity_g}
\end{equation}

If the inequality $\theta<\zeta/2 \wedge\tilde{\zeta}$ is satisfied by:
\begin{align}
\zeta & =\frac{1+\rho_{\min}(A^{\mathrm{T}}A)}{\rho_{\mathrm{max}}(\Sigma)}-\lambda_{g},\label{eq:example_zeta}\\
\tilde{\zeta} & =\frac{1}{\rho_{\mathrm{max}}(\Sigma)}\wedge\left\{ \frac{1}{\rho_{\mathrm{max}}(\Sigma_{0})}+\frac{\rho_{\min}(A^{\mathrm{T}}A)}{\rho_{\mathrm{max}}(\Sigma)}\right\} -\lambda_{g},\label{eq:example_zeta_tilde}\\
\theta & =\frac{\rho_{\mathrm{max}}(A^{\mathrm{T}}A)^{1/2}}{\rho_{\min}(\Sigma)},\label{eq:example_theta}
\end{align}
then Condition \ref{assu:hmm} holds.
\end{lemma}
\noindent The proof is in section appendix \ref{sec:discussion_proofs}. We pass the following remarks on this result:
\begin{itemize}
\item The condition (\ref{eq:semi_log_concavity_g}) is called semi-log-concavity
of $x\mapsto g(x,y)$, generalizing log-concavity by allowing
$\lambda_{g}\in\mathbb{R}$, rather than only $\lambda_{g}\leq0$. 
\item The condition $\theta<\zeta/2\wedge\tilde{\zeta}$ can be interpreted
as balancing the magnitude of temporal correlation in (\ref{eq:linear_gauss_signal})
against the fluctuations of $W_{n}$ and the degree to
which the mapping $x\mapsto g(x,y)$ is informative
about $x$. As $\lambda_{g}\to-\infty$ the mapping $x\mapsto g(x,y)$
becomes more strongly log-concave, and by inspection of (\ref{eq:example_zeta})-(\ref{eq:example_theta})
the condition $\theta<\zeta/2\wedge\tilde{\zeta}$ can always be achieved
if $\lambda_{g}$ takes a negative value large enough in magnitude, with other
quantities on the right of the equations (\ref{eq:example_zeta})-(\ref{eq:example_theta})
held constant.  On the other hand, if $\rho_{\mathrm{max}}(\Sigma)^{-1}\wedge\rho_{\mathrm{max}}(\Sigma_{0})^{-1}>\lambda_{g}$,
which implies $\zeta\wedge\tilde{\zeta}>0$ for any value of $\rho_{\min}(A^{\mathrm{T}}A)$,
the condition $\theta<\zeta/2\wedge\tilde{\zeta}$ can be achieved if
$\rho_{\mathrm{max}}(A^{\mathrm{T}}A)$ is small enough.


\item The fact that $\zeta$, $\tilde{\zeta}$ and $\theta$ in (\ref{eq:example_zeta})-(\ref{eq:example_theta})
depend only on eigenvalues of $A$, $\Sigma$ and $\Sigma_{0}$ and
the semi-concavity parameter $\lambda_{g}$ means they, and consequently
$\lambda$ and $\gamma$, do not necessarily depend on dimension.
As a simple example consider the case: $\lambda_{g}\leq0$, $A=aI_{d}$
and $\Sigma=\sigma^{2}I_{d}$, with $|a|<1$ and $\sigma^{2}>0$.
In this situation $\theta<\zeta/2\wedge\tilde{\zeta}$ holds, and $\gamma$ and $\lambda$ can be chosen to depend only on $|a|$ and $\sigma^{2}$ and be such that  (\ref{eq:lambda_conditions}) holds.
\end{itemize}


\begin{lemma}\label{lem:non_gauss}
Assume:

\noindent a) the signal satisfies
\begin{equation}
X_{n}=A(X_{n-1})+W_{n},\label{eq:linear_gauss_signal-1}
\end{equation}
where for $n\in\mathbb{N}$, $W_{n}$ is independent of other random
variables and has density proportional to $e^{-\psi(w)}$, and $\mu(x)\propto e^{-\psi_{0}(x)}$,
where $\psi,\psi_{0}:\mathbb{R}^{d}\to\mathbb{R}$ are continuously
differentiable and have bounded, Lipschitz gradients in the sense that:

\noindent
\begin{align*}
 & \sup_{w}\|\nabla\psi_{0}(w)\|\vee\|\nabla\psi(w)\|\leq L_{\psi},\\
 & \|\nabla\psi(w)-\nabla\psi(w^{\prime})\|\vee\|\nabla\psi_{0}(w)-\nabla\psi_{0}(w^{\prime})\|\leq L_{\nabla\psi}\|w-w^{\prime}\|,
\end{align*}
$A$ is continuously differentiable, and has bounded, Lipschitz gradient
in the sense that:
\[
\sup_{x}\|\nabla A(x)\|_{\mathrm{op}}\leq L_{A},\quad\|\nabla A(x)-\nabla A(x^{\prime})\|_{\mathrm{op}}\leq L_{\nabla A}\|x-x^{\prime}\|,
\]
where $\nabla A(x)$ is the Jacobian matrix of $x\mapsto A(x)$, and
$\|\cdot\|_{\mathrm{op}}$ is the Euclidean operator norm.

\noindent b) for each $y\in\mathbb{Y}$, $x\mapsto g(x,y)$ is strictly
positive, continuously differentiable and there exists $\lambda_{g}<0$
such that for all $x,x^\prime\in\mathbb{R}^d$, $y\in\mathbb{Y}$,
\begin{equation}
\left\langle x-x^{\prime},\nabla_{x}\log g(x,y)-\nabla_{x}\log g(x^{\prime},y)\right\rangle \leq\lambda_{g}\|x-x^{\prime}\|^{2}.\label{eq:semi_log_concavity_g-1}
\end{equation}
If $\theta<\zeta/2\wedge\tilde{\zeta}$ is satisfied with:
\begin{align*}
\zeta & =\tilde{\zeta}=-(L_{\nabla\psi}+L_{A}^{2}L_{\nabla\psi}+L_{\psi}L_{\nabla A})-\lambda_{g},\\
\theta & =L_{\nabla\psi}L_{A},
\end{align*}
then Condition \ref{assu:hmm} holds.
\end{lemma}
The proof is in section appendix \ref{sec:discussion_proofs}.

\subsection{Behaviour of $\alpha_{\gamma,n}(y)$ and $\beta_n(y)$ as $n\to\infty$}\label{sec:behaviour_of_alpha_and_beta}
Lemma \ref{lem:poly_moment_growth} and Lemma \ref{lem:alpha_and_beta} are intermediate results which will be applied to bound the right hand sides of 
\eqref{eq:thm_1_bound} and \eqref{eq:thm_2_bound}. The proofs of these two lemmas are in appendix \ref{sec:discussion_proofs}.

\begin{lemma}\label{lem:poly_moment_growth}
For any nonnegative random variables $(Z_{n})_{n\geq 0}$ , if there exists $p>0$ and $s>0$ such that
$$
\sup_{n\geq0}\frac{\mathbf{E}[Z_n^s]}{(n+1)^p}<\infty,
$$
then for any $\rho\in(0,1)$,
\[
\sup_{n\geq0}\rho^{n}Z_{n}<\infty,\quad a.s.
\]
\end{lemma}

\begin{lemma}\label{lem:alpha_and_beta}
If $Y=(Y_n)_{n\geq 0}$ is any $\mathbb{Y}$-valued stochastic process with law denoted $\mathbf{P}_Y$ such that for some $s\in(0,1]$ and $p>0$
$$
\sup_{n\geq0} \frac{1}{(n+1)^p} \mathbf{E}\left[ \|\left.\nabla_{x}\log g(x,Y_{n})\right|_{x=0}\|^{2s}\right] <\infty,
$$
then for any $\rho\in(0,1)$ and any $\gamma >0$,
$$
\sup_{n\geq 0 } \rho^n \alpha_{\gamma,n}(y)<\infty\quad\text{and}\quad \sup_{n\geq 0 } \rho^n \beta_{n}(y)<\infty,\quad \text{for}\;\mathbf{P}_Y\text{-almost all }y.
$$
For any $\tilde{\rho}\in (\rho,1)$, there exist $C(y)$ and $D(y)$ such that $C(y)\vee D(y)<\infty$ for $\mathbf{P}_Y$-almost all $y$, and 
$$
\rho^n  \alpha_{\gamma,n}(y) \leq \tilde{\rho}^n C(y),\quad\text{and}\quad \sum_{k=n}^{\infty} \rho^k \beta_{k}(y) \leq \tilde{\rho}^n D(y),\quad   \text{for all } n \text{ and }\mathbf{P}_Y \text{-almost all }y.
$$
\end{lemma}

\subsection{An example with a Gaussian signal and heavy-tailed observations}\label{sec:heavy_tailed}

In this section we  consider an instance of the model class addressed in Lemma \ref{lem:AR1} and the application of Theorem \ref{thm:main_non_uniform} to it.  With $d=1$ and $\mathbb{Y}=\mathbb{R}$ we consider
\begin{equation}
X_n = A X_{n-1} +W_n,\qquad Y_n = X_n + V_n,\label{eq:students_t_hmm}
\end{equation}
where $A\in\mathbb{R}$, $|A|<1$, $W_n\sim \mathcal{N}(0,1)$, and $V_n$ follows a Student's-$t$ distribution with one degree of freedom:
\begin{equation}
g(x,y)=\pi^{-1/2}\{1+(y-x)^{2}\}^{-1}\label{eq:student_t}.
\end{equation} 
The mapping $x\mapsto g(x,y)$ is not log-concave,  so (\ref{eq:semi_log_concavity_g}) is not satisfied with any $\lambda_{g}\leq0$, but it is satisfied for some 
 $\lambda_{g}>0$. Indeed we have:
$$
\frac{\partial}{\partial x} \log g(x,y) = \frac{-2(x-y)}{1+(x-y)^2},\qquad \frac{\partial^2}{\partial x^2} \log g(x,y)= -2\frac{1-2(x-y)^2}{[1+(x-y)^2]^2},
$$
from which follows that $\frac{\partial^2}{\partial x^2} \log g(x,y) \leq 2$, uniformly in $x$ and $y$, and in turn that (\ref{eq:semi_log_concavity_g}) holds with $\lambda_g \coloneqq 2$.

For \eqref{eq:students_t_hmm} with $d=1$, the matrix $A$ in the signal model \eqref{eq:linear_gauss_signal} is reduced to a single real-valued number and $\rho_{\mathrm{max}}(A^{\mathrm{T}}A)=A^2$. Thus in accordance with the discussion following Lemma \ref{lem:AR1}, if we take the scalar $|A|<1$ small enough, then $\theta<\zeta/2\wedge\tilde{\zeta}$ is achieved and thus by Lemma \ref{lem:AR1}, Condition \ref{assu:hmm} holds and Theorem \ref{thm:main_non_uniform} may be applied. 

Our next step is to obtain a simplified bound for the right-hand side of the bound \eqref{eq:thm_1_bound} from Theorem \ref{thm:main_non_uniform}. By direct calculations we find:
\begin{equation}
\left[\left.\frac{\partial}{\partial x} \log g(x,y)\right|_{x=0}\right]^2\leq 4|y_n|^2,\label{eq:students_t_example_beta_alpha}
\end{equation}
there exists a finite constant $\chi$ such that:
$$
|\nabla_n\phi_n(x)|^2 \vee|\nabla_n\tilde{\phi}_n(x)|^2 \leq \chi (|x_{n-1}|^2 +|x_{n}|^2 + |x_{n+1}|^2) + 4|y_n|^2
$$
and
$$\beta_n(y)\leq4|y_n|^2,\qquad \eta_n(r,y)\leq \frac{r}{\gamma} + 4|y_n|^2.$$
From \eqref{eq:thm_1_bound} we obtain:
\begin{multline}
\sup_{m\in\mathbb{N}_{0},m\geq n}\|\xi_{n}-\xi_{m}\|_{\gamma}^{2}\leq \frac{ \gamma^{n}}{\lambda^{2}}
\left(\chi\lambda^{-2}\frac{\alpha_{\gamma,n}(y)}{\gamma} + 4|y_n|^2\right)
\\+\frac{\gamma^{n+1}}{\lambda^2}\left(\chi\lambda^{-2}\alpha_{\gamma,n}(y) + 4|y_{n+1}|^2\right)
+\frac{1}{\lambda^{2}}\sum_{k=n+2}^{\infty}\gamma^{k}\beta_k(y).\label{eq:student_t_example_bound1}
\end{multline}

We now seek assumptions on the observations under which the right hand side of \eqref{eq:student_t_example_bound1} converges to zero as $n\to\infty$, in which case $(\xi_n)_{n\geq0}$ is a Cauchy sequence. We first consider the case of a well-specified model,  where $\mathbf{P}_Y$ is the law of the random variables $(Y_n)_{n\geq0}$ distributed according to  \eqref{eq:students_t_hmm}. We approach this using Lemma \ref{lem:alpha_and_beta} and Lemma \ref{lem:poly_moment_growth}.

From \eqref{eq:students_t_hmm}, using  $|A|<1$ and the distributional assumptions on $(W_n)_{n\geq 0}$ and $(V_n)_{n\geq 0}$, we have for any $s\in(0,1/2)$, $(|a|+|b|)^{2s}\leq |a|^{2s} + |b|^{2s}$,
$$
\sup_{n\geq 0}\mathbf{E}[|Y_n|^{2s}] \leq \sup_{n\geq 0}\mathbf{E}[|X_n|^{2s}] + \mathbf{E}[|V_0|^{2s}] <\infty,
$$
and using \eqref{eq:students_t_example_beta_alpha}
\begin{equation}
\sup_{n\geq 0} \mathbf{E}\left[\left(\left.\frac{\partial}{\partial x} \log g(x,Y_n)\right|_{x=0}\right)^{2s}\right] \leq 4^s \sup_{n\geq0}\mathbf{E}[|Y_n|^{2s}]<\infty.\label{eq:sup_Exp_students_t_example}
\end{equation}
Having established \eqref{eq:sup_Exp_students_t_example} we may apply Lemma \ref{lem:alpha_and_beta} to bound $\alpha_{\gamma,n}(y)$ and $\sum_{k\geq{n+1}}\gamma^k \beta_{k}(y)$  and Lemma \ref{lem:poly_moment_growth} with there $Z_n=|Y_n|^{2}$ to bound $\gamma^n |Y_n|^{2}$. We thus find that for any $\tilde{\gamma}\in(\gamma,1)$ there exists $C(y)$ such that $C(y)<\infty$ for $\mathbf{P}_Y$-almost all $y$, and the right hand side of  \eqref{eq:student_t_example_bound1} is   bounded by $C(y)\tilde{\gamma}^n$  for $\mathbf{P}_Y$-almost all $y$. This implies that there exists a set in $\mathbb{Y}^\star_\gamma\subset\mathbb{Y}^{\mathbb{N}}$ such that  under \eqref{eq:students_t_hmm}, $\mathbf{P}_Y(\mathbb{Y}^\star_\gamma)=1$, for any $y\in\mathcal{Y}$, the right hand side of \eqref{eq:student_t_example_bound1} converges to zero as $n\to\infty$, and $\xi_n$ is thus a Cauchy sequence.

Concerning the misspecified case, suppose that $Y=(Y_n)_{n\geq0}$ does not necessarily follow \eqref{eq:students_t_hmm} but rather is any $\mathbb{R}$-valued stochastic process, with law denoted $\mathbf{P}_Y$, such that for some $p>0$ and $s>0$:
$$
\sup_{n\geq0} \frac{\mathbf{E}[|Y_n|^{2s}]}{(1+n)^p}<\infty.
$$
Then by Lemma \ref{lem:poly_moment_growth}, $\sup_{n\geq0}\gamma^n |y_n|^2 < \infty$, for $\mathbf{P}_Y$-almost all $y$; via \eqref{eq:students_t_example_beta_alpha}, 
$$
\sup_{n\geq 0} \mathbf{E}\left[\left(\left.\frac{\partial}{\partial x} \log g(x,Y_n)\right|_{x=0}\right)^{2s}\right] < \infty,
$$
and Lemma \ref{lem:alpha_and_beta} can be applied to again show that the right hand side of  \eqref{eq:student_t_example_bound1} is a  bounded by $C(y)\tilde{\gamma}^n$ for all $n\geq0$, $\mathbf{P}_Y$-almost all $y$, $\tilde{\gamma}\in(\gamma,1)$ and some $C(y)<\infty$.




\subsection{An example with a nonlinear, nonstationary, non-Gaussian signal} \label{sec:non_stationary}

In this section we  consider an instance of the model class addressed in Lemma \ref{lem:non_gauss} and the application of Theorem \ref{thm:main_uniform} to it. With $\mathbb{Y}=\mathbb{R}^d$, the model of interest is:
\begin{equation}
X_n = A(X_{n-1})+W_n\qquad Y_n = X_n + V_n,\label{eq:nonlinear_HMM}
\end{equation}
where $A$ and $(W_n)_{n\geq0}$ are taken to satisfy the assumptions of Lemma \ref{lem:non_gauss} and $V_n \sim \mathcal{N}(0,\sigma_y^2 I_n)$, so that in Lemma \ref{lem:non_gauss} we can take $\lambda_g = -\sigma_y^{-2}$. 

In the case $d=1$, an example of a function $\psi$ (or $\psi_{0}$) which is continuously differentiable and has a bounded, Lipschitz gradient in the sense of Lemma \ref{lem:non_gauss}  is the Huber function, with some $c>0$:
\[
\psi(x)=\begin{cases}
\frac{1}{2c}x^{2}, & |x|\leq c\\
|x|-\frac{c}{2}, & |x|>c.
\end{cases}
\]

By taking $\sigma_y^2$ small enough the condition $\theta < \zeta/2 \wedge \tilde{\zeta}$ in Lemma \ref{lem:non_gauss} is satisfied and therefore Condition \ref{assu:hmm} holds. We now turn to verification of Condition \ref{assu:hmm2}. By direct calculation:
\begin{equation}
\beta_n(y)=\|\left.\nabla_x \log g(x,y_n)\right|_{x=0}\|^2  =\|y_n\|^2 /\sigma_y^4. \label{eq:nonlinear_example_beta}
\end{equation}
and, using the same estimates as in the proof of Lemma \ref{lem:non_gauss},  there exists a finite constant $c$ such that for any $y\in\mathbb{Y}$,
\begin{multline}
\|\nabla_n\phi_n(x,y)-\nabla_n\phi_n(x^\prime,y)\| \vee \|\nabla_n\tilde{\phi}_n(x,y)-\nabla_n\tilde{\phi}_n(x^\prime,y)\|
\\
\leq c \left(\|x_{n-1} - x_{n-1}^\prime \| + \|x_{n} - x_{n}^\prime \| + \|x_{n+1} - x_{n+1}^\prime \|\right),\label{eq:nonlinear_example_phi-phi}
\end{multline}
and assumption a) of Condition \ref{assu:hmm2} holds for all $y\in\mathbb{Y}^{\mathbb{N}}$.

Now let $\mathbb{Y}^\star_\gamma \coloneqq \{y=(y_0,y_1,\ldots)\in\mathbb{Y}^{\mathbb{N}}:\sum_n \gamma^n \|y_n\|^2 <\infty\}$. Considering \eqref{eq:nonlinear_example_beta}, clearly assumption b) of Condition \ref{assu:hmm2} holds for any $y\in\mathbb{Y}^\star_\gamma$. It remains to verify assumption c) of Condition \ref{assu:hmm2}. To do so we combine the following identity:
\begin{align*}
\|\partial U(x,y) - \partial U(x^\prime,y)\|_\gamma^2 &= \|\nabla_0\tilde{\phi}_0(x,y)-\nabla_0\tilde{\phi}_0(x^\prime,y)\|^2 \\
&+\sum_{n\geq 1} \gamma^n \|\nabla_n\phi_n(x,y)-\nabla_n\phi_n(x^\prime,y)\|^2
\end{align*}
with \eqref{eq:nonlinear_example_phi-phi}.

We have thus established that Condition \ref{assu:hmm2} holds and hence Theorem \ref{thm:main_uniform} holds for any $y\in\mathbb{Y}^\star_\gamma$. Our next objective is to exhibit conditions under which $\mathbf{P}_Y(\mathbb{Y}^\star_\gamma)=1$ and establish the rate at which the right hand side of the bound \eqref{eq:thm_2_bound} from Theorem  \ref{thm:main_uniform} converges to zero exponentially hast as $n\to\infty$, for $\mathbf{P}_Y$-almost all $y$. 

Consider the well-specified case, that is where $\mathbf{P}_Y$ is the law of the observations $(Y_n)_{n\geq0}$ corresponding to \eqref{eq:nonlinear_HMM}. To control the growth of moments of $\|Y_n\|$ over time and apply Lemma  \ref{lem:poly_moment_growth}  we shall assume that $\|A(x)-A(x^\prime)\|\leq \|x-x^\prime\|$, in which case $\|A(x)\|\leq \|A(0)\|+\|x\|$,
$$
\|X_n\| \leq \|A(X_{n-1})\| +  \|W_n\| \leq  \|A(0)\|+ \|X_{n-1}\| + \|W_n\| \leq n\|A(0)\| + \sum_{m=1}^n \|W_n\| +\|X_0\| 
$$
and for any $s\in(1/2,1)$, by Minkowski's inequality,
\begin{align}
\mathbf{E}[\|Y_n\|^{2s}]^{1/{2s}} &\leq \mathbf{E}[\|X_n\|^{2s}]^{1/{2s}} + \mathbf{E}[\|V_0\|^{2s}]^{1/{2s}}  \nonumber\\
&\leq n \|A(0)\| + n\mathbf{E}[ \|W_1\|^{2s}]^{1/{2s}} + \mathbf{E}[\|X_0\|^{2s} ]^{1/{2s}} + \mathbf{E}[\|V_0\|^{2s} ]^{1/{2s}}.\label{eq:nongaussian_example_Y_moment}
\end{align}
This estimate allows us to apply Lemma \ref{lem:poly_moment_growth} with $Z_n$ there taken to be $\|Y_n\|^2$, and combined with \eqref{eq:nonlinear_example_beta} we obtain for any $\tilde{\gamma}\in(\gamma,1)$,
\begin{align*}
\sum_{n=0}^\infty \gamma^n \beta_n(y) &\leq \sigma_y^{-4} \sum_{n=0}^\infty \gamma^n \|y_n\|^2 \\
&\leq \sigma_y^{-4} \left(\sup_{n\geq 0} \tilde{\gamma}^n\|y_n\|^2 \right)\sum_{n=0}^\infty \frac{\gamma^n}{\tilde{\gamma}^n} <\infty, \quad \text{for }\mathbf{P}_Y\text{-almost all }y. 
\end{align*}
Therefore under \eqref{eq:nonlinear_HMM}, $\mathbf{P}_Y(\mathbb{Y}^\star_\gamma)=1$. The bound \eqref{eq:nongaussian_example_Y_moment} also allows Lemma \ref{lem:alpha_and_beta} to be applied, and we conclude that there exists $C(y)$ such that $C(y)<\infty$ for $\mathbf{P}_Y$-almost all $y$, and the right hand side of  \eqref{eq:thm_2_bound} is bounded by $C(y)\tilde{\gamma}^n$ for all $n\geq0$ and $\mathbf{P}_Y$-almost all $y$.

Again using Lemma \ref{lem:alpha_and_beta},  we reach the same conclusion in the misspecified case if $\mathbf{P}_Y$ is the law of any  $\mathbb{R}^d$-valued stochastic process $Y=(Y_n)_{n\geq0}$ such that for some $p>0$ and $s>0$,
$$
\sup_{n\geq0}\frac{\mathbf{E}[\|Y_n\|^s]}{(1+n)^p}<\infty.
$$

\subsection{Parallelized approximate optimization}\label{sec:parallel}
The existence of the limit in (\ref{eq:Viterbi}), as has been verified via Theorems \ref{thm:main_non_uniform} and \ref{thm:main_uniform} in the preceeding sections,  suggests that (\ref{eq:optimization_problem})
can be solved approximately using a collection of optimization algorithms
which process data segments in parallel. With
$\Delta$ and $\ell=(n+1)/\Delta$ assumed to be integers, consider the partition:
\[
\{0,\ldots,n\}=\bigcup_{k=1}^{\ell}A_{k},\qquad A_{k}\coloneqq\{(k-1)\Delta,\ldots,k\Delta-1\},
\]
and for an integer $\delta>0$ consider the $\delta$-enlargement
of each $A_{k}$,
\begin{equation}
A_{k}(\delta)\coloneqq\{m\in\{0,\ldots,n\}:\exists a\in A_{k}:|m-a|\leq\delta\}.\label{eq:A_enlarged}
\end{equation}
Suppose the $\ell$ optimization problems:
\begin{equation}
\argmax_{x_{A_{k}(\delta)}}\;p(x_{A_{k}(\delta)}|y_{A_{k}(\delta)}),\qquad k=1,\ldots,\ell,\label{eq:paralell_opt_problem}
\end{equation}
where $x_{A_{k}(\delta)}=(x_{m};m\in A_{k}(\delta))$, $y_{A_{k}(\delta)}=(y_{m};m\in A_{k}(\delta))$, are
solved in parallel. Then in a post-processing step, for each $k$,
the components indexed by $A_{k}(\delta)\setminus A_{k}$ of the solution
to $\argmax_{x_{A_{k}(\delta)}}\;p(x_{A_{k}(\delta)}|y_{A_{k}(\delta)})$
are discarded, and what remains concatenated across $k$ to give an
approximation to the solution of (\ref{eq:optimization_problem}). If it takes $T(n)$ time to solve (\ref{eq:optimization_problem})
the speed-up from parallelization could be as much as a factor of
$T(n)/T(\Delta+2\delta)$.  The Euclidean norm of the approximation error associated with the first segment in the parallelization scheme
can be bounded using (\ref{eq:thm_1_bound}) or
(\ref{eq:thm_2_bound}). The following is an immediate corollary of (\ref{eq:thm_2_bound}).
\begin{corollary}\label{cor:parallel}
If the assumptions of Theorem \ref{thm:main_uniform} hold for some given $y\in\mathbb{Y}^{\mathbb{N}}$,
\begin{align}
\sup_{n\geq\Delta+\delta}\sum_{m=0}^{\Delta}\|\xi_{\Delta+\delta,m}-\xi_{n,m}\|^{2}&\leq\gamma^{-\Delta}\sup_{n\geq\Delta+\delta}\|\xi_{\Delta+\delta}-\xi_{n}\|_{\gamma}^{2}\nonumber\\
&\leq\frac{1}{\lambda^{2}}\left(\gamma^{\delta-1}\frac{2\chi}{\lambda^{2}}\alpha_{\gamma,\Delta+\delta}(y)+\sum_{k=\delta}^{\infty}\gamma^{k}\beta_{\Delta+k}(y)\right).\label{eq:error bound}
\end{align}
\end{corollary}
The right hand side of this bound can be controlled using the same tools and techniques demonstrated in sections \ref{sec:heavy_tailed} and \ref{sec:non_stationary} to analyse its convergence to zero as $\delta\to \infty$.

\subsection{Application to a model of neural population activity}\label{subsec:neural_model}

State-space models are used in neuroscience to examine
time-varying dependence in the firing activity of neural populations  and  have been advocated for use in detecting
cell assemblies in the brain -- ensembles of neurons exhibiting
coordinated firing -- thought to play a key role in memory
formation and learning \cite{shimazaki2012state,donner2017approximate}. Neural spiking data in the form of multivariate binary time series are commonly modelled using random fields with time-dependent parameters. Here $y_{n}=(y_{n,k}^{(i)},1\leq i\leq N,1\leq k\leq R)\in\{0,1\}^{NR}=:\mathbb{Y}$,
where $y_{n,k}^{(i)}\in\{0,1\}$ indicates absence or presence of spiking activity of the $i$'th of
$N$ neurons during the $n$'th time bin of the $k$'th of $R$ replicated
experimental trials.

Similarly to \cite{shimazaki2012state} we consider a random field model for $y_{n}$
given $x_{n}$, where the latent state has components $x_{n}=(x_{n}^{(i,j)},1\leq i<j\leq N)\in\mathbb{R}^{N(N-1)/2}$,
that is $d=N(N-1)/2$., and:
\begin{equation}
g(x_{n},y_{n})=\exp\left\{\sum_{j>i}\frac{x_{n}^{(i,j)}}{R}\sum_{k=1}^{R}(y_{n,k}^{(i)}-c^{(i)})(y_{n,k}^{(j)}-c^{(j)})-C(x_{n})\right\}.\label{eq:neuro_like-1}
\end{equation}
Where $c^{(i)}$ is the average firing rate of the $i$th neuron over
the $R$ trials. The interest in this model is that the variables $x_{n}^{(i,j)}$ can be interpreted as a time-dependent statistical coupling between the firing activity of neurons $i$ and $j$. The normalizing factor $C(x_{n})$ is too expensive
to compute for anything more than a handful of neurons and following
\cite{donner2017approximate} we consider the pseudo-likelihood approximation:
\begin{align}
\tilde{g}(x_{n},y_{n})&=\prod_{i=1}^{N}\prod_{k=1}^{R}\frac{\exp(y_{n,k}^{(i)}z_{n,k}^{(i)})}{1+\exp(y_{n,k}^{(i)}z_{n,k}^{(i)})},\label{eq:neuro_ps_like-1}\\
z_{n,k}^{(i)}&=\frac{1}{R}\left\{ \sum_{j<i}x_{n}^{(j,i)}(y_{n,k}^{(j)}-c^{(j)})+\sum_{i<j}x_{n}^{(i,j)}(y_{n,k}^{(j)}-c^{(j)})\right\} .\nonumber
\end{align}
Both (\ref{eq:neuro_like-1}) and (\ref{eq:neuro_ps_like-1}) are
log-concave functions of $x_{n}$. Combining either with a prior model
for the signal process as in (\ref{eq:linear_gauss_signal}) with
$b=0$, $A=a I_{N(N-1)/2}$, $|a|<1$, $\Sigma=\sigma^{2}I_{N(N-1)/2}$, where $I_{N(N-1)/2}$ is the identity matrix of size $N(N-1)/2$
and the prior distribution for $X_{0}$, $\mu$, set to the stationary
distribution of (\ref{eq:linear_gauss_signal}) hence by Lemma \ref{lem:AR1},  Condition \ref{assu:hmm} holds and
one may take $\gamma$, $\lambda$
in Theorem \ref{thm:main_non_uniform} independently of $d$ and hence $N$.

For both (\ref{eq:neuro_like-1}) and (\ref{eq:neuro_ps_like-1}) the fact that $\mathbb{Y}=\{0,1\}^{NR}$ is a set with finitely many elements, combined with the linear-Gaussian nature of the signal implies that for 
any $y=(y_{n})_{n\geq 0}\in \mathbb{Y}^{\mathbb{N}}$,
$\sup_{n}\beta_n(y)<\infty$, $\sup_{n}\alpha_{n,\gamma}(y)<\infty$,
and $\sup_{n}\eta_{n}(r,y)<\infty$ for any $r$. Therefore applied to either of these two models, for \emph{any} $y\in\mathbb{Y}^{\mathbb{N}}$, the right hand side of the bound \eqref{eq:thm_1_bound} from Theorem \ref{thm:main_non_uniform} converges to zero as $n\to\infty$. 


\subsection{Numerical results}
We consider neural recordings of action potential spike trains from $30$ medial prefrontal cortical neurons. The data were recorded using a 384-electrode Neuropixels probe \cite{jun2017fully} while an adult male rat navigated a 3-arm maze. Each recording was over a duration $35$ seconds, from -15 to +20 seconds around the rat's arrival at particular location on the maze called the ``reward point''. The presence or absence of spiking per neuron was recorded in bins of width $10$ millisecond, so that $n=3500$. The data were divided in two subsets. The first consisted of $R=76$ replicates of ``correct'' trials in which the rat navigated a particular route on the maze and consequently received a sweet reward at the reward point. The second data subset consisted of $R=22$ ``error'' trials, in which the rat did not take that route and consequently received no reward. The scientific objective in analyzing these data is to study time-varying statistical interactions across the population of neurons as the rat approaches and passes through the reward point, and to identify differences in these interactions between the ``correct'' and ``error'' trials. We consider the pseudo-likelihood (\ref{eq:neuro_ps_like-1}) combined with signal model described below that equation, and since there are $N=30$ neurons we have $d=435$.

Our first objective is to numerically evaluate the error and speed-up associated with the parallelization scheme.  For this purpose, gradient-descent was used to approximately solve each of the $\ell$ optimization problems \eqref{eq:paralell_opt_problem} for the neural model described above applied to the data subset of ``correct'' trials. The $\ell$ instances of gradient-descent were implemented in parallel using MATLAB's ``parfor'' command across $\ell$ Intel Sandy Bridge cores, each running at $2.6$ Ghz, on a single blade of the University of Bristol's BlueCrystal Phase 3 cluster. $1.5\times10^4$ iterations of gradient descent were performed with a constant step size $10^{-8}$ in each instance. Let $\{\hat{\xi}_{n,0}(\ell,\delta),\ldots,\hat{\xi}_{n,n}(\ell,\delta)\}$ be the resulting approximation to \eqref{eq:optimization_problem} obtained by combining the approximate solutions to \eqref{eq:paralell_opt_problem} as described in section \ref{subsec:Background-and-motivation}.  Figure \ref{fig1} shows the relative error:
\begin{equation}
\frac{\sqrt{\sum_{m=0}^n\|\hat{\xi}_{n,m}(\ell,\delta)-\hat{\xi}_{n,m}(1,0)\|^2}}{\sqrt{\sum_{m=0}^n\|\hat{\xi}_{n,m}(1,0)\|^2}}\label{eq:rel_error}
\end{equation}
against the overlap parameter $\delta$. Results are shown for the full data set from $N=30$ neurons and also for a subset consisting of the first $N=5$ neurons. Here $\Delta$ is determined by $\Delta=(n+1)/\ell$. The parameters of the state-space model were set to $a=0.95$ and $\sigma = 10^{-4}$.

Corollary \ref{cor:parallel} suggests that \eqref{eq:rel_error} should decay to zero exponentially fast as $\delta$ grows. This is apparent in figure \ref{fig1}.  A reduction in error as $\ell$ decreases can also be observed. The plot on the left of figure \ref{fig2} shows the same results as figure \ref{fig1} but with relative error on a logarithmic scale. Since the lines in figure \ref{fig2} are close to parallel it appears that there is no degradation with $N$ (hence $d$) in the exponential rate at which the relative error decays as $\delta$ grows. The plot on the right of figure \ref{fig2} also shows the relative execution time of the parallelized scheme, that is the time taken to compute $\{\hat{\xi}_{n,m}(\ell,\delta),\ldots,\hat{\xi}_{n,n}(\ell,\delta)\}$ using $\ell=2,4,\ldots,16$ cores divided by the time to compute $\{\hat{\xi}_{n,0}(1,0),\ldots,\hat{\xi}_{n,n}(1,0)\}$. Moving from $1$ to $2$ and $4$ cores results in a roughly linear speed-up. Beyond $4$ cores the speed-up is sublinear, which may be due to communication overhead associated with parallelization. From figure \ref{fig2}, $\delta=100$ and $\ell=4$ results in a relative error of less that $0.1\%$ and a speed-up from parallelization of $1/0.32 = 3.125$.
\begin{figure}
\centering
\includegraphics[width=0.9\columnwidth]{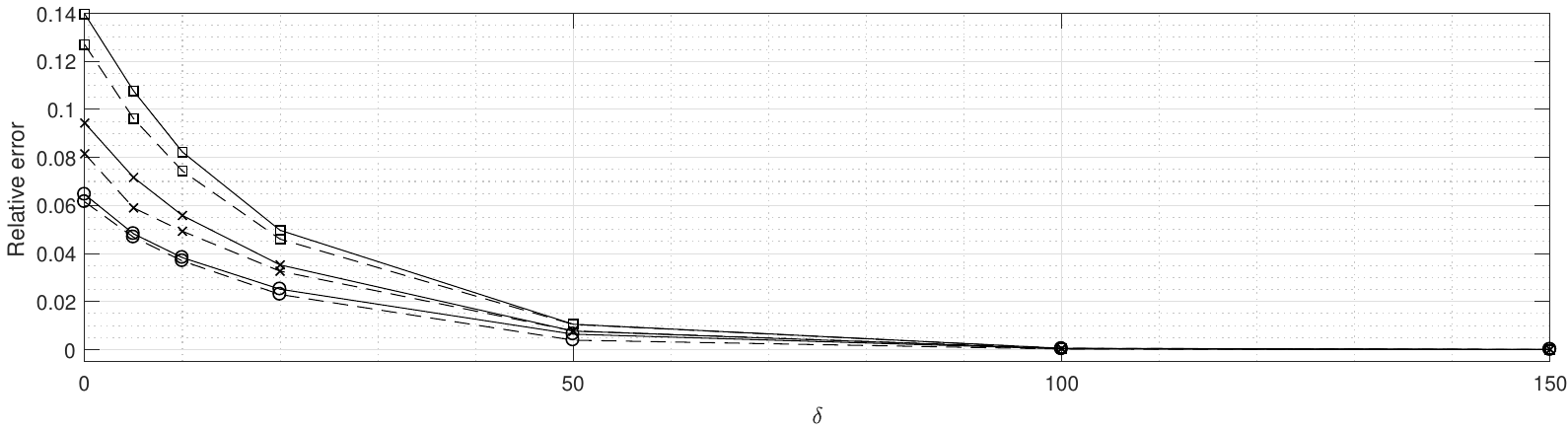}
\caption{Relative error in the parallelization scheme vs. the overlap parameter $\delta$. $N=5$ neurons (dash), $N=30$ neurons (solid), $\ell=4$ cores ($\Box$), $\ell=8$ cores ($\times$) and $\ell=6$ cores ($\circ$).}
\label{fig1}
\end{figure}

\begin{figure}
\centering
\includegraphics[width=0.9\columnwidth]{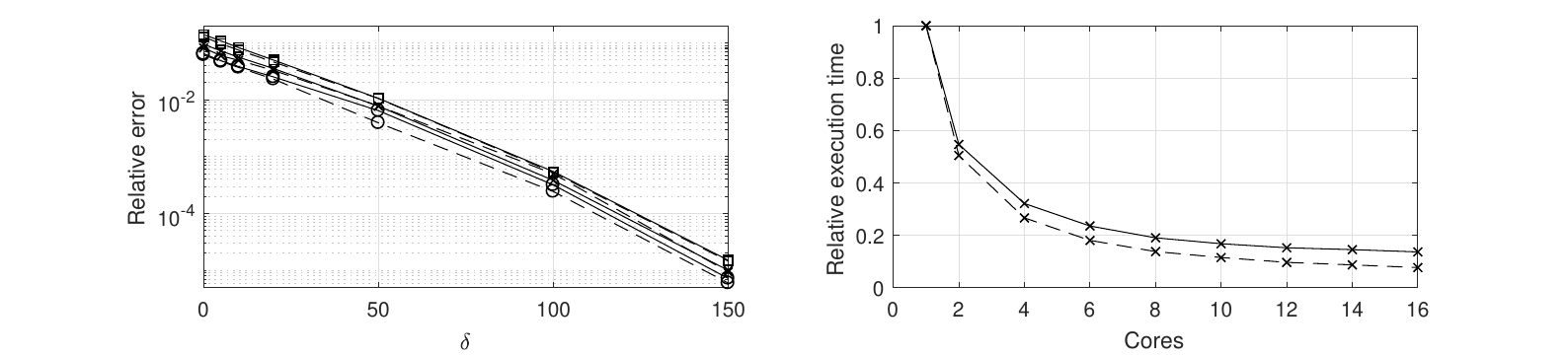}
\caption{Left: relative error vs. overlap parameter $\delta$ as in figure \ref{fig1} but on logarithmic scale. Right: Relative execution time vs. number of cores. $\delta=100$ (solid), $\delta=10$, (dashed).}
\label{fig2}
\end{figure}
Our next objective is to interpret the state-estimates obtained for the two data subsets. Figure \ref{fig3} shows maximum a-posteriori state estimates of the pairwise coupling parameters $x^{ij}_n$ for the two data subsets, consisting respectively of ``correct'' and ``error'' trials. The first and fourth rows in figure \ref{fig3} display $\hat{\xi}_{n,m}(4,100)$ as a heat map, for five different time steps: $m$ corresponding to $3$ ($t=-3$) and $1$ ($t=-1$) seconds before arrival at the reward point; at the reward point ($t=0$), and $1$ ($t=1$) and $3$ ($t=3$) seconds after pass the reward point. These times are marked by vertical blue lines in the second and third row plots in figure \ref{fig3}.

Signatures of neural population interaction emerge from the estimates of pairwise coupling parameters: on correct trials, the red coloring on the heat-maps corresponds to strong positive influence from a small minority of neurons onto a larger pool, at $t=-3$, $t=1$, $t=0$ and extending to $t=1$. Such anticipatory activity could reflect reward expectancy. Conversely, on ``error'' trials the predominantly blue colouring on the heatmaps at $t=-1$, $t=0$ and $t=1$ indicates negative pairwise interactions shortly before, during and after the reward point. The estimates on the ``error'' trials also present greater variation over time than in the ``correct'' trials,  possibly reflecting an error signal or the consolidation of trial outcome-related information.

\begin{figure}
\centering
\includegraphics[width=0.9\columnwidth]{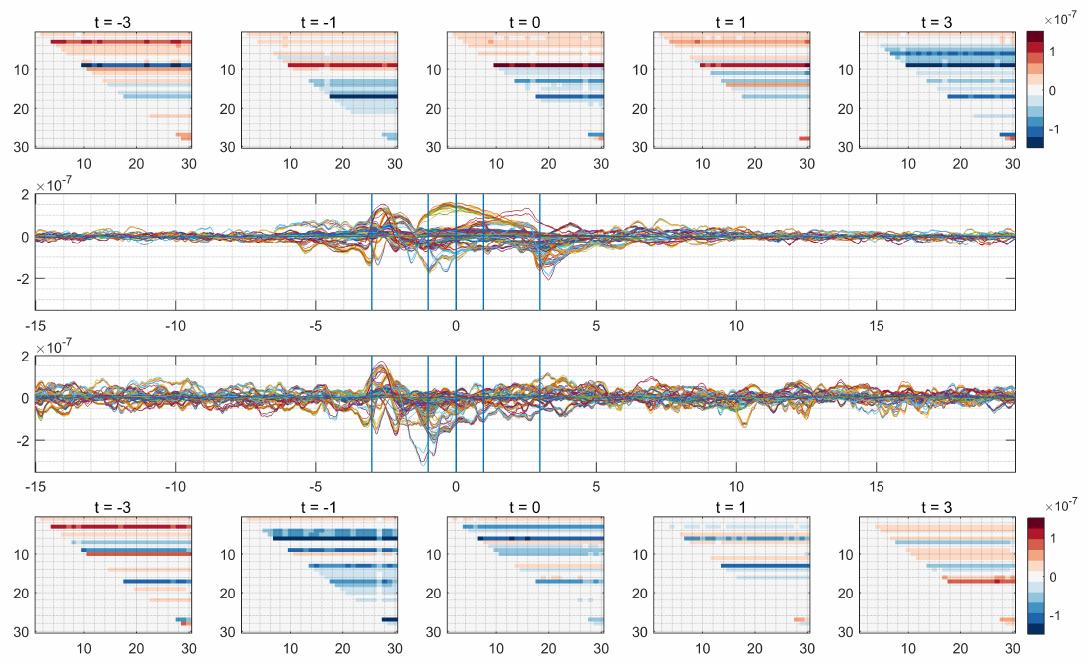}
\caption{First and second rows: ``correct'' trials, third and fourth rows: ``error'' trials. Second and third row show $\{\hat{\xi}_{n,0}(4,100),\ldots,\hat{\xi}_{n,n}(4,100)\}$ with time on the horizontal axis in seconds relative to the reward delivery time. First and fourth rows display $\hat{\xi}_{n,m}(4,100)$ as a heat map, for $m$ corresponding left-to-right to: $3$ and $1$ seconds before reward delivery, delivery time itself, and $1$ and $3$ seconds after reward delivery. These times are marked by the vertical blue lines in the second and third row plots.}
\label{fig3}
\end{figure}

\subsection{Comparison to the assumptions of \cite{chigansky2011viterbi}}\label{subsec:Comparison-to-the}
The assumptions of \cite[Thm 3.1]{chigansky2011viterbi} require
that $x\mapsto\mu(x)$ and $(x,x^\prime)\mapsto f(x,x^\prime)$ are log-concave,
and that $x\mapsto g(x,y)$ is strongly log-concave, uniformly in
$y$. As per section \ref{subsec:verifying},
our Condition \ref{assu:hmm}b) does not require all these conditions
to hold simultaneously. Assumption (a4) (sic.) of \cite[Thm 3.1]{chigansky2011viterbi} is
that with $f(x,x^\prime)\propto e^{-\alpha(x,x^\prime)}$, there is a non-decreasing
function $g:\mathbb{R}_{+}\to\mathbb{R}_{+}$ growing to $+\infty$
not faster than polynomially, such that for all $M>0$,
\begin{equation}
\alpha(x,x^\prime)\leq M\quad\Longrightarrow\quad\left|\frac{\partial^{2}}{\partial x\partial x^\prime}\alpha(x,x^\prime)\right|\leq g(M),\quad\text{for all }x,x^\prime.\label{eq:chigansky}
\end{equation}
Putting aside the issue of once versus twice differentiability, this
assumption is related to the terms multiplied by $\theta$ in our
Condition \ref{assu:hmm}b), but allows greater generality because
$g(M)$ can grow with $M$, where as our Condition \ref{assu:hmm}b)
requires a value of $\theta$ uniform in $x,x^{\prime}$. \cite[Thm 3.1]{chigansky2011viterbi}
also places an assumption on the asymptotic behaviour of $n^{-1}U_n(X_{0},\ldots,X_{n},Y_0,\ldots,Y_n)$
as $n\to\infty$  which may be regarded as a counterpart to the assumptions
of Lemma \ref{lem:alpha_and_beta}.
%
%


%



\begin{appendix}

\section{Fr\'echet derivatives and differential equations}

\subsection{Fr\'echet derivatives\label{subsec:Fr=0000E9chet-derivatives}}
The following definitions can be found in \cite[App. A]{hopper2011ricci}.
For Banach spaces $V,W$ over $\mathbb{R}$, with respective norms
$\|\cdot\|_{V}$, $\|\cdot\|_{W}$, a function $\varphi:V\to W$ has
a directional derivative at $x\in V$ in direction $v\in V$ if there
exists $\partial\varphi(v;x)\in W$ such that
\[
\lim_{\epsilon\to0}\left\Vert \frac{\varphi(x+\epsilon v)-\varphi(x)}{\epsilon}-\partial\varphi(v;x)\right\Vert _{W}=0.
\]
The function $\varphi$ is G\^ateaux differentiable at $x$ if $\partial\varphi(v;x)$
exists for all $v\in V$ and $D\varphi(\cdot;x):v\mapsto\partial\varphi(v;x)$
is a bounded linear operator from $V$ to $W$, in which case $D\varphi(\cdot;x)$
is called the G\^ateaux derivative at $x$. The function $\varphi$
is additionally Fr\'echet differentiable at $x$ if

\begin{equation}
\lim_{\epsilon\to0}\sup_{v:\|v\|_{V}=1}\left\Vert \frac{\varphi(x+\epsilon v)-\varphi(x)}{\epsilon}-\partial\varphi(v;x)\right\Vert _{W}=0,\label{eq:frechet_defn}
\end{equation}
in which case the operator $D\varphi(\cdot;x)$ is called the Fr\'echet
derivative at $x$.

\subsection{Ordinary differential equations on the Hilbert space\label{subsec:ODE's-on-the-Hilbert}}

In the following proposition the operator of orthogonal projection
from $l_{2}(\gamma)$ to $l_{2}^{n}(\gamma)$ is written $\Pi_n$.
\begin{proposition}
\label{prop:ODE}For a given triple $(\gamma,F,n)$ consisting of
a constant $\gamma\in(0,1]$, a mapping $F:l_{2}(\gamma)\to l_{2}(\gamma)$
and $n\geq 0\cup\{\infty\}$, assume that a)-c) hold:

\noindent a) $F$ is continuous with respect to the norm $\|\cdot\|_{\gamma}$
on $l_{2}(\gamma)$,

\noindent b) there exists $\lambda>0$ such that for all $x,x^{\prime}\in l_{2}^{n}(\gamma)$,
\[
\left\langle x-x^{\prime},F(x)-F(x^{\prime})\right\rangle _{\gamma}\leq-\lambda\|x-x^{\prime}\|_{\gamma}^{2},
\]

\noindent c) $F(x)=F\circ\Pi_{n}(x)$  and $F(x)\in l_{2}^{n}(\gamma)$ for all $x\in l_{2}(\gamma)$.

Then there exists a globally defined and unique flow $\Phi:(t,x)\in\mathbb{R}_{+}\times l_{2}(\gamma)\mapsto\Phi(t,x)\in l_{2}(\gamma)$
solving the Fr\'echet differential equation,
\[
\frac{\mathrm{d}}{\mathrm{d}t}\Phi(t,x)=F(\Phi(t,x)),\qquad\Phi(0,x)=x.
\]
 This flow has a unique in $l_{2}^{n}(\gamma)$ fixed point, $\xi$, and this point satisfies $F(\xi)=0$
and $\|\Phi(t,x)-\xi\|_{\gamma}\leq e^{-\lambda t}\|x-\xi\|_{\gamma}$
for all $x\in l_{2}^{n}(\gamma)$ and $t\geq0$.
\end{proposition}
\noindent The proof is postponed.

The term $\frac{\mathrm{d}}{\mathrm{d}t}\Phi(t,x)$ in Proposition
\ref{prop:ODE} is an application of the Fr\'echet derivative of $\Phi(t,x)$
with respect to $t$, that is in (\ref{eq:frechet_defn}), $V$ is
$\mathbb{R}$ equipped with the Euclidean norm, $W$ is the Hilbert
space $l_{2}(\gamma)$, and $\varphi$ is the map $t\mapsto\Phi(t,x)$,
where in the latter the $x$ argument is regarded as fixed. Similarly
with $x$ fixed, and denoting the Fr\'echet derivative of $t\mapsto\Phi(t,x)$
at $t$ by $D\Phi(\cdot;t,x)$, the quantity $\frac{\mathrm{d}}{\mathrm{d}t}\Phi(t,x)$
is precisely $D\Phi(1;t,x)$. Thus in particular,
\[
\lim_{\delta\searrow0}\left\Vert \frac{\Phi(t+\delta,x)-\Phi(t,x)}{\delta}-\frac{\mathrm{d}}{\mathrm{d}t}\Phi(t,x)\right\Vert _{\gamma}=0,
\]
which, in general, is stronger than the element-wise convergence
of $[\Phi(t+\delta,x)-\Phi(t,x)]/\delta$ to $\frac{\mathrm{d}}{\mathrm{d}t}\Phi(t,x)$.

The following lemma will be used in the proof of Proposition \ref{prop:ODE}.
\begin{lemma}
\label{lem:chain rule}If a triple $(\gamma,F,n)$ satisfies the assumptions
of Proposition \ref{prop:ODE}, then with $\Phi$ as therein and any
$x,x^{\prime}\in l_{2}(\gamma)$,
\begin{align}
\frac{\mathrm{d}}{\mathrm{d}t}\|\Phi(t,x)-\Phi(t,x^{\prime})\|_{\gamma}^{2} & =2\left\langle \Phi(t,x)-\Phi(t,x^{\prime}),F(\Phi(t,x))-F(\Phi(t,x^{\prime}))\right\rangle _{\gamma}.\label{eq:dt_of_square_norm_of_phi_diff}
\end{align}
\end{lemma}
\begin{proof}
In the case $n<\infty$, assumption c) of Proposition \ref{prop:ODE}
implies that only the first $d(n+1)$ elements of the vector $\Phi(t,x)$
depend on $t$, and in that case the lemma can be proved by the chain
rule of elementary differential calculus. The following proof is valid
for any $n\geq 0\cup\{\infty\}$ and uses the chain rule
of Fr\'echet differentiation.

Pick any $x,v\in l_{2}(\gamma)$, write them as $x=[x_{0}^{\mathrm{T}}\,x_{1}^{\mathrm{T}}\,\cdots]^{\mathrm{T}}$,
$v=[v_{0}^{\mathrm{T}}\,v_{1}^{\mathrm{T}}\,\cdots]^{\mathrm{T}}$ with each $x_{k},v_{k}\in\mathbb{R}^{d}$.
The first step is to prove that the mapping $\varphi(x)=\|x\|_{\gamma}^{2}$
is Fr\'echet differentiable everywhere in $l_{2}(\gamma)$, with Fr\'echet
derivative $D\varphi(v;x)=2\left\langle v,x\right\rangle _{\gamma}$.

Consider the existence of directional derivatives. For $m\in\mathbb{N}$
let $e_{m}$ denote the vector in $l_{2}(\gamma)$ whose $m$th entry
is $1$ and whose other entries are zero. The directional derivative
$\partial\varphi(e_{m};x)$ clearly exists.

We now need to check the existence of directional derivatives of $\varphi$
in arbitrary directions in $l_{2}(\gamma)$. To do so we shall validate
the following four equalities:
\begin{align}
\lim_{\epsilon\to0}\frac{\varphi(x+\epsilon v)-\varphi(x)}{\epsilon} & =\lim_{\epsilon\to0}\lim_{m\to\infty}\frac{\varphi(x+\epsilon\Pi_{m}(v))-\varphi(x)}{\epsilon}\label{eq:chain_rule1}\\
 & =\lim_{m\to\infty}\lim_{\epsilon\to0}\frac{\varphi(x+\epsilon\Pi_{m}(v))-\varphi(x)}{\epsilon}\label{eq:chain_rule2}\\
 & =\lim_{m\to\infty}2\sum_{k=0}^{m}\gamma^{k}\left\langle v_{k},x_{k}\right\rangle \label{eq:chain_rule3}\\
 & =2\left\langle v,x\right\rangle _{\gamma}.\label{eq:chain_rule_4}
\end{align}

For (\ref{eq:chain_rule1}), we have for any $\epsilon>0$,
\begin{align}
|\varphi(x+\epsilon\Pi_{m}(v))-\varphi(x+\epsilon v)| & \leq\sum_{k\geq m+1}\gamma^{k}\left|\|x_{k}\|^{2}-\|x_{k}+\epsilon v_{k}\|^{2}\right|\nonumber \\
 & \leq3\sum_{k\geq m+1}\gamma^{k}\|x_{k}\|^{2}+\epsilon^{2}\sum_{k\geq m+1}\gamma^{k}\|v_{k}\|^{2}\nonumber \\
 & \to0,\quad\mathrm{as}\quad m\to\infty,\label{eq:v^m_to_v}
\end{align}
where the convergence holds since $x$ and $v$ are members of $l_{2}(\gamma)$.

Let $\nabla_{k}\varphi(x)$ be the vector in $\mathbb{R}^{d}$ whose
$i$th entry is the partial derivative of $\varphi(x)$ with respect
to the $i$th element of $x_{k}$, that is $\nabla_{k}\varphi(x)=2x_{k}$.
Since $\varphi(x)=\sum_{k=0}^{\infty}\gamma^{k}\|x_{k}\|^{2}$, the
directional derivative in direction $\Pi_{m}(v)$ at $x$ is given
by:
\begin{equation}
\partial\varphi(\Pi_{m}(v);x)=\lim_{\epsilon\to0}\frac{\varphi(x+\epsilon\Pi_{m}(v))-\varphi(x)}{\epsilon}=\sum_{k=0}^{m}\left\langle v_{k},\nabla_{k}\varphi(x)\right\rangle =2\sum_{k=0}^{m}\gamma^{k}\left\langle v_{k},x_{k}\right\rangle .\label{eq:truncated_derivatives}
\end{equation}
Let us now check that the convergence in (\ref{eq:truncated_derivatives})
is uniform in $m$ in order to verify the equality in (\ref{eq:chain_rule2}).
By the mean value theorem of elementary differential calculus, for
any $\epsilon>0$ there exists $y^{m,\epsilon}$ on the line segment
between $x$ and $x+\epsilon\Pi_{m}(v)$ (so $y_{k}^{m,\epsilon}=x_{k}$
for $k>m$) such that
\begin{align}
&\sup_{m}\left|\frac{\varphi(x+\epsilon\Pi_{m}(v))-\varphi(x)}{\epsilon}-2\left\langle \Pi_{m}(v),x\right\rangle _{\gamma}\right| \nonumber\\
& =\sup_{m}\left|\frac{\sum_{k=0}^{m}\epsilon\left\langle v_{k},\nabla_{k}\varphi(y^{m,\epsilon})\right\rangle }{\epsilon}-2\left\langle \Pi_{m}(v),x\right\rangle _{\gamma}\right|\nonumber \\
 & =2\sup_{m}\left|\left\langle \Pi_{m}(v),y^{m,\epsilon}-x\right\rangle _{\gamma}\right|\nonumber \\
 & \leq2\sup_{m}\|\Pi_{m}(v)\|_{\gamma}\|y^{m,\epsilon}-x\|_{\gamma}\nonumber \\
 & \leq2\|v\|_{\gamma}^{2}\epsilon,\label{eq:uniform convergence}
\end{align}
so the convergence in (\ref{eq:truncated_derivatives}) is indeed
uniform in $m$. Therefore (\ref{eq:chain_rule2}) holds.

For the two remaining equalities, (\ref{eq:chain_rule3}) is already
proved in (\ref{eq:truncated_derivatives}), and (\ref{eq:chain_rule_4})
holds by Cauchy-Schwartz combined with the facts that $x,v\in l_{2}(\gamma)$
and that absolute convergence of a series in $\mathbb{R}$ implies
its convergence.

We have established that the directional derivative of $\varphi$
at an arbitrary $x$ in an arbitrary direction $v$ exists and is
given by $2\left\langle v,x\right\rangle _{\gamma}$. To prove that
$\varphi$ is everywhere G\^ateaux differentiable, we also need to show
that for each $x$, $D\varphi(\cdot;x):v\mapsto2\left\langle v,x\right\rangle _{\gamma}$
is a bounded operator from $l_{2}(\gamma)$ to $\mathbb{R}$. This
follows from Cauchy-Schwartz:
\[
\sup_{v\neq0}\frac{2|\left\langle v,x\right\rangle _{\gamma}|}{\|v\|_{\gamma}}\leq2\|x\|_{\gamma}<+\infty,\quad\text{for all }x\in l_{2}(\gamma).
\]
To prove that $\varphi$ is Fr\'echet differentiable everywhere in $l_{2}(\gamma)$,
it suffices, by \cite[App A, Prop A.3]{hopper2011ricci}, to check
that $D\varphi(\cdot;x)$ is operator-norm continuous in $x$. This
follows again by the Cauchy-Schwartz inequality:
\[
\sup_{v\neq0}\frac{2|\left\langle v,x\right\rangle _{\gamma}-\left\langle v,x^{\prime}\right\rangle _{\gamma}|}{\|v\|_{\gamma}}=\sup_{v\neq0}\frac{2|\left\langle v,x-x^{\prime}\right\rangle _{\gamma}|}{\|v\|_{\gamma}}\leq2\|x-y\|_{\gamma},\quad\text{for all }x,x^{\prime}\in l_{2}(\gamma).
\]

We have proved that $\varphi(x)=\|x\|_{\gamma}^{2}$ is Fr\'echet differentiable
everywhere in $l_{2}(\gamma)$, with Fr\'echet derivative in direction
$v$ given by $D\varphi(v;x)=2\left\langle v,x\right\rangle _{\gamma}$.

The proof is completed by an application of the chain rule of Fr\'echet
differentiation:
\begin{align*}
\frac{\mathrm{d}}{\mathrm{d}t}\|\Phi(t,x)-\Phi(t,x^{\prime})\|_{\gamma}^{2} & =D\varphi\left(\frac{\mathrm{d}}{\mathrm{d}t}\{\Phi(t,x)-\Phi(t,x^{\prime})\};\Phi(t,x)-\Phi(t,x^{\prime})\right)\\
 & =2\left\langle F\{\Phi(t,x)\}-F\{\Phi(t,x^{\prime})\},\Phi(t,x)-\Phi(t,x^{\prime})\right\rangle _{\gamma}.
\end{align*}
\end{proof}
\begin{proof}
[Proof of Proposition \ref{prop:ODE}]Let $n\geq 0\cup\{\infty\}$ be as in the statement of the proposition.  Applying assumptions b) and c) of the proposition, we have for any
$x,x^{\prime}\in l_{2}(\gamma)$,
\begin{align*}
\left\langle F(x)-F(x^{\prime}),x-x^{\prime}\right\rangle _{\gamma} & =\left\langle F\circ\Pi_{n}(x)-F\circ\Pi_{n}(x^{\prime}),\Pi_{n}(x)-\Pi_{n}(x^{\prime})\right\rangle _{\gamma}\\
 & +\left\langle F\circ\Pi_{n}(x)-F\circ\Pi_{n}(x^{\prime}),x-\Pi_n(x)-x^{\prime}+\Pi_n(x^{\prime})\right\rangle _{\gamma}\\
 & \leq-\lambda\|\Pi_{n}(x)-\Pi_{n}(x^{\prime})\|_{\gamma}^{2}+0\\
 & \leq0.
\end{align*}
This global dissipation condition, combined with assumption a) of
the proposition allows the application of \cite[Thm 3.4, p.41]{deimling2006ordinary}
on the Hilbert space $l_{2}(\gamma)$ to give the existence and uniqueness
of the globally defined flow as required.

Under assumption
c) of the proposition, $F(x)\in l_{2}^{n}(\gamma)$ for any $x\in l_{2}(\gamma)$.Then, since $\Phi(t,x)=x+\int_0^t F(\Phi(s,x))\mathrm{d}s$,we find that if $x\in l_{2}^{n}(\gamma)$,
then $\Phi(t,x)\in l_{2}^{n}(\gamma)$ for all $t>0$. Now fix any $x,x^{\prime}\in l_{2}^n(\gamma)$ and define $a(t)\coloneqq \|\Phi(t,x)-\Phi(t,x^\prime)\|_\gamma^2$.  Lemma \ref{lem:chain rule} combined with assumption b) of the proposition gives:
$$
\frac{\mathrm{d}}{\mathrm{d}t}a(t)\leq -2\lambda a(t),
$$
from which it follows that
$$
a(t)\leq a(0)\exp(-2\lambda t).\label{eq:a(t)}
$$ 
We have thus proved that $\|\Phi(t,x)-\Phi(t,x^{\prime})\|_{\gamma}\leq e^{-\lambda t}\|x-x^{\prime}\|_{\gamma}$
for all $x,x^{\prime}\in l_{2}^{n}(\gamma)$.  An application
of the Banach fixed point theorem to the restriction of $\Phi$ to
the Hilbert space $l_{2}^{n}(\gamma)$ then gives the existence of
the unique (in $l_{2}^{n}$) fixed point $\xi$. Since $\xi$ is a fixed point we have for all $t>0$
$$
\xi = \Phi(t,\xi) =  \xi+ \int_0^t F(\Phi(s,\xi)) \mathrm{d}s =\xi+  F(\xi) t,
$$
which implies $ F(\xi)=0$.
\end{proof}

\section{Proofs for section \ref{sec:quantitative_bounds}}\label{subsec:Proofs-of-the}
In Lemma \ref{lem:check_dissipative_cond} and the proof of Theorem
\ref{thm:main_non_uniform} below we shall need the following generalization
of the inner-product $\left\langle \cdot,\cdot\right\rangle _{\gamma}$
and norm $\|\cdot\|_{\gamma}$, for $n\geq 0$,
\begin{equation}
\left\langle x,x^{\prime}\right\rangle _{\gamma,n}=\sum_{m=0}^{\infty}\gamma^{|m-n|}\left\langle x_{m},x_{m}^{\prime}\right\rangle ,\qquad\|x\|_{\gamma,n}=\left\langle x,x\right\rangle _{\gamma,n}^{1/2},\qquad x,x^{\prime}\in l_{2}(\gamma).\label{eq:defn_n_prod}
\end{equation}

\begin{lemma}
\label{lem:check_dissipative_cond}Assume that Condition \ref{assu:hmm}
holds, and with $\zeta,\tilde{\zeta},\theta$ as therein and $\gamma\in(0,1]$
assume the following inequalities hold:
\begin{equation}
\zeta>\frac{\theta}{2\gamma}(1+\gamma)^{2},\qquad\tilde{\zeta}>\frac{\theta}{2\gamma}(1+\gamma).\label{eq:theta_gamma_inequality-1}
\end{equation}
Then for any $\lambda$ such that:
\begin{equation}
0<\lambda\leq\left\{ \tilde{\zeta}-\frac{\theta}{2\gamma}(1+\gamma)^{2}\right\} \wedge\left\{ \zeta-\frac{\theta}{2\gamma}(1+\gamma)\right\} ,\label{eq:lambda_conditions-1}
\end{equation}
any $n\geq 0$ and $m=0,\ldots,n$,
\[
\left\langle x-x^{\prime},\nabla U_n(x,y)-\nabla U_n(x^{\prime},y)\right\rangle _{\gamma,m}\geq\lambda\|x-x^{\prime}\|_{\gamma,m}^{2},\quad\text{for all}\quad x,x^{\prime}\in l_{2}^{n}(\gamma).
\]
\end{lemma}
\begin{proof}
For any $x,x^{\prime}\in l_{2}^{n}(\gamma)$ we have:
\begin{align*}
 & \left\langle x-x^{\prime},\nabla U_n(x,y)-\nabla U_n(x^{\prime},y)\right\rangle _{\gamma,m}\nonumber \\
 & =-\gamma^{|m|}\left\langle x_{0}-x_{0}^{\prime},\nabla_{0}\tilde{\phi}_{0}(x)-\nabla_{0}\tilde{\phi}_{0}(x^{\prime})\right\rangle \nonumber \\
 & \quad-\sum_{k=1}^{n-1}\gamma^{|k-m|}\left\langle x_{k}-x_{k}^{\prime},\nabla_{k}\phi_{k}(x)-\nabla_{k}\phi_{k}(x^{\prime})\right\rangle \nonumber \\
 & \quad-\gamma^{|n-m|}\left\langle x_{n}-x_{n}^{\prime},\nabla_{n}\tilde{\phi}_n(x,y)-\nabla_{n}\tilde{\phi}_{n}(x^{\prime},y)\right\rangle \nonumber 
    \end{align*}
\begin{align}
& \geq\gamma^{m}\left\{ \tilde{\zeta}\|x_{0}-x_{0}^\prime \|^{2}-\theta\|x_{0}-x_{0}^{\prime}\|\|x_{1}-x_{1}^{\prime}\|\right\} \nonumber \\
 & \quad+\sum_{k=1}^{n-1}\gamma^{|k-m|}\left\{ \zeta\|x_{k}-x_{k}^{\prime}\|^{2}-\theta\|x_{k}-x_{k}^{\prime}\|\left(\|x_{k-1}-x_{k-1}^{\prime}\|+\|x_{k+1}-x_{k+1}^{\prime}\|\right)\right\} \nonumber \\
 & \quad+\gamma^{n-m}\left\{ \tilde{\zeta}\|x_{n}-x_{n}^{\prime}\|^{2}-\theta\|x_{n}-x_{n}^{\prime}\|\|x_{n-1}-x_{n-1}^{\prime}\|\right\} \nonumber \\
& \geq\gamma^{m}\left\{ (\tilde{\zeta}-\frac{\theta}{2})\|x_{0}-x_{0}^{\prime}\|^{2}-\frac{\theta}{2}\|x_{1}-x_{1}^{\prime}\|^{2}\right\} \nonumber \\
 & \quad+\sum_{k=1}^{n-1}\gamma^{|k-m|}\left\{ (\zeta-\theta)\|x_{k}-x_{k}^{\prime}\|^{2}-\frac{\theta}{2}\|x_{k-1}-x_{k-1}^{\prime}\|^{2}-\frac{\theta}{2}\|x_{k+1}-x_{k+1}^{\prime}\|^{2}\right\} \nonumber \\
 & \quad+\gamma^{n-m}\left\{ (\tilde{\zeta}-\frac{\theta}{2})\|x_{n}-x_{n}^{\prime}\|^{2}-\frac{\theta}{2}\|x_{n-1}-x_{n-1}^{\prime}\|^{2}\right\}\\
 & =\gamma^{m}\left\{ (\tilde{\zeta}-\frac{\theta}{2})\|x_{0}-x_{0}^{\prime}\|^{2}-\frac{\gamma^{|1-m|}}{\gamma^{m}}\frac{\theta}{2}\|x_{0}-x_{0}^{\prime}\|^{2}\right\} \nonumber \\
 & \quad\sum_{k=1}^{n-1}\gamma^{|k-m|}\left\{ (\zeta-\theta)\|x_{k}-x_{k}^{\prime}\|^{2}-\frac{\gamma^{|k-1-m|}}{\gamma^{|k-m|}}\frac{\theta}{2}\|x_{k}-x_{k}^{\prime}\|^{2}-\frac{\gamma^{|k+1-m|}}{\gamma^{|k-m|}}\frac{\theta}{2}\|x_{k}-x_{k}^{\prime}\|^{2}\right\} \nonumber \\
 & \quad\gamma^{n-m}\left\{ (\tilde{\zeta}-\frac{\theta}{2})\|x_{n}-x_{n}^{\prime}\|^{2}-\frac{\gamma^{|n-1-m|}}{\gamma^{n-m}}\frac{\theta}{2}\|x_{n}-x_{n}^{\prime}\|^{2}\right\} \nonumber \\
 & \geq\gamma^{m}\{\tilde{\zeta}-\frac{\theta}{2}(1+\gamma^{-1})\}\|x_{0}-x_{0}^{\prime}\|^{2}\nonumber \\
 & \quad+\{\zeta-\theta-\frac{\theta}{2}(\gamma+\gamma^{-1})\}\sum_{k=1}^{n-1}\gamma^{|k-m|}\|x_{k}-x_{k}^{\prime}\|^{2}\nonumber \\
 & \quad+\{\tilde{\zeta}-\frac{\theta}{2}(1+\gamma^{-1})\}\gamma^{n-m}\|x_{n}-x_{n}^{\prime}\|^{2}\nonumber \\
 & \geq\lambda\sum_{k=0}^{n}\gamma^{|k-m|}\|x_{k}-x_{k}^{\prime}\|^{2},\label{eq:verify_convexity-1}
\end{align}
 where the first equality and inequality are due to (\ref{eq:gradU^N_defn})
and Condition \ref{assu:hmm}b); the second inequality uses a re-arrangement of the terms in the summation and the fact
that for any $a,b\in\mathbb{R}$, $2|a||b|\leq|a|^{2}+|b|^{2}$; the
third inequality uses $\frac{\gamma^{|1-m|}}{\gamma^{m}}\vee\frac{\gamma^{|n-1-m|}}{\gamma^{n-m}}\leq\frac{1}{\gamma}$
for $0\leq m\leq n$ and $\frac{\gamma^{|k-1-m|}}{\gamma^{|k-m|}}+\frac{\gamma^{|k+1-m|}}{\gamma^{|k-m|}}\leq\frac{1}{\gamma}+\gamma$
; and the final inequality holds under the conditions on $\lambda,\zeta,\tilde{\zeta},\theta$
and $\gamma$ given in (\ref{eq:theta_gamma_inequality}) and (\ref{eq:lambda_conditions}).
\end{proof}
\begin{proof}
[Proof of Theorem \ref{thm:main_non_uniform}] Throughout the proof,
$n\geq 0$ and $\gamma\in(0,1]$ are fixed. Considering
$(\gamma,-\nabla U_n,n)$, let us validate the assumptions of Proposition
\ref{prop:ODE} in the order: c), then a), then b). Assumption c)
of Proposition \ref{prop:ODE} holds due to the definition of $\nabla U_n$
in (\ref{eq:gradU^N_defn}). Validating assumption a) of Proposition
\ref{prop:ODE} requires that if $x\to x^{\prime}$ in $l_{2}(\gamma)$
then $\|\nabla U_n(x,y)-\nabla U_n(x^{\prime},y)\|_{\gamma}\to0$.
But we have already validated assumption c) of Proposition \ref{prop:ODE},
so $\nabla U_n$ maps $l_{2}(\gamma)$ into $l_{2}^{n}(\gamma)$,
and
\begin{align*}
\|\nabla U_n(x,y)-\nabla U_n(x^{\prime},y)\|_{\gamma}^{2} & =\|\nabla_{0}\tilde{\phi}_{0}(x,y)-\nabla_{0}\tilde{\phi}_{0}(x^{\prime},y)\|^{2}\\
&+\sum_{m=1}^{n-1}\gamma^{m}\|\nabla_{m}\phi_{m}(x,y)-\nabla_{m}\phi_{m}(x^{\prime},y)\|^{2}\\
 & +\gamma^{n}\|\nabla_{n}\tilde{\phi}_n(x,y)-\nabla_{n}\tilde{\phi}_{n}(x^{\prime},y)\|^{2}.
\end{align*}
Also by assumption c) of Proposition \ref{prop:ODE}, $\nabla U_n(x,y)$
depends on $x$ only through $(x_{0},\ldots,x_{n})$. These observations
together with Condition \ref{assu:hmm}a) validate assumption a) of
Proposition \ref{prop:ODE}. Assumption b) of Proposition \ref{prop:ODE}
holds by an application of Lemma \ref{lem:check_dissipative_cond}.
This completes the verification of the assumptions of Proposition
\ref{prop:ODE} for $(\gamma,-\nabla U_n,n)$ and thus establishes
the existence of the fixed point $\xi_{n}$.

Our next step is to obtain bounds on $\|\xi_{n}\|_{\gamma,n}^{2}$
and $\|\xi_{n}\|_{\gamma,n+1}^{2}$. An application of Lemma \ref{lem:check_dissipative_cond}
and Cauchy-Schwartz gives:
\[
\|\xi_{n}\|_{\gamma,n}^{2}\leq\frac{1}{\lambda}\left\langle 0-\xi_{n},\nabla U_n(0,y)-0\right\rangle _{\gamma,n}\leq\frac{1}{\lambda}\|\xi_{n}\|_{\gamma,n}\|\nabla U_n(0,y)\|_{\gamma,n},
\]
hence
\begin{equation}
\|\xi_{n}\|_{\gamma,n}^{2}\leq\frac{\alpha_{\gamma,n}(y)}{\lambda^{2}}\quad\text{and}\quad\|\xi_{n}\|_{\gamma,n+1}^{2}=\gamma\|\xi_{n}\|_{\gamma,n}^{2}\leq\gamma\frac{\alpha_{\gamma,n}(y)}{\lambda^{2}},\label{eq:alpha_bound_on_xi}
\end{equation}
where the equality uses the fact that $\|\xi_{n,m}\|=0$ for $m>n$.

Now fix any $m>n$ . An application of Lemma \ref{lem:check_dissipative_cond}
and Cauchy-Schwartz gives:
\begin{equation}
\|\xi_{n}-\xi_{m}\|_{\gamma,0}^{2}\leq\frac{1}{\lambda}\left\langle \xi_{n}-\xi_{m},\nabla U_{m}(\xi_{n},y)-0\right\rangle _{\gamma,0}\leq\frac{1}{\lambda}\|\xi_{n}-\xi_{m}\|_{\gamma,0}\|\nabla U_{m}(\xi_{n},y)\|_{\gamma,0}.\label{eq:cs_proof}
\end{equation}
Observe $\nabla U_n(\xi_{n},y)=0$ implies  $\nabla_{0}\tilde{\phi}_{0}(\xi_{n},y)=\nabla_{k}\phi_{k}(\xi_{n},y)=0$
for all $k<n$.  Combining this fact with $\xi_{n}\in l_{2}^{n}(\gamma), $(\ref{eq:gradU^N_defn}),
(\ref{eq:eta_defn}) and the bound (\ref{eq:alpha_bound_on_xi})
gives:
\begin{align}
\|\nabla U_{m}(\xi_{n},y)\|_{\gamma,0}^{2} & =\sum_{k=n}^{m-1}\gamma^{k}\|\nabla\phi_{k}(\xi_{n},y)\|^{2}+\gamma^{m}\|\nabla\tilde{\phi}_{m}(\xi_{n},y)\|^{2}\nonumber \\
 & \leq\sum_{k=n}^{n+1}\gamma^{k}\|\nabla\phi_{k}(\xi_{n},y)\|^{2}+\sum_{k=n+2}^{\infty}\gamma^{k}\beta_{k}(y)\nonumber \\
 & \leq\gamma^{n}\eta_{n}\left(\frac{\alpha_{\gamma,n}(y)}{\lambda^{2}},y\right)+\gamma^{n+1}\eta_{n+1}\left(\gamma\frac{\alpha_{\gamma,n}(y)}{\lambda^{2}},y\right)+\sum_{k=n+2}^{\infty}\gamma^{k}\beta_{k}(y).\label{eq:F^m_bound_proof}
\end{align}
The proof of the theorem is completed by combining this bound with
(\ref{eq:cs_proof}).
\end{proof}
\begin{proof}
[Proof of Theorem \ref{thm:main_uniform}]

The first step is to apply Proposition \ref{prop:ODE} to $(\gamma,-\partial U(\cdot,y),\infty)$.
From its definition (\ref{eq:grad_U_infty_defn}) combined with assumptions
b) and c) of Condition \ref{assu:hmm2}, it is clear that $\partial U(\cdot,y)$ maps
$l_{2}(\gamma)$ into itself and $\Pi_{\infty}=\mathrm{Id}$ by definition,
so assumption c) of Proposition \ref{prop:ODE} is satisfied. Assumption
c) of Condition \ref{assu:hmm2} is exactly what is required for assumption a) of
Proposition \ref{prop:ODE} to hold. Let us now verify assumption
b) of Proposition \ref{prop:ODE}. For any$x,x^{\prime}\in l_{2}(\gamma)$
and $n\geq 0,$
\begin{equation}
\left\langle x-x^{\prime},\partial U(x,y)-\partial U(x^{\prime},y)\right\rangle _{\gamma}=\left\langle x-x^{\prime},\nabla U_n(x,y)-\nabla U_n(x^{\prime},y)\right\rangle _{\gamma}+\Delta_{n}(x,x^{\prime})\label{eq:verify_convexity_inf}
\end{equation}
where
\begin{align}
|\Delta_{n}(x,x^{\prime})| & =|\left\langle x-x^{\prime},\partial U(x,y)-\nabla U_n(x,y)+\nabla U_n(x^{\prime},y)-\partial U(x^{\prime},y)\right\rangle _{\gamma}|\nonumber \\
 & \leq\|x-y\|_{\gamma}\{\|\nabla U_n(x,y)-\partial U(x,y)\|_{\gamma}+\|\nabla U_n(x^{\prime},y)-\partial U(x^{\prime},y)\|_{\gamma}\}.\label{eq:delta_(x,y)}
\end{align}
Using the facts that $\nabla U_n(\cdot,y)=\nabla U_n(\cdot,y)\circ\Pi_{n}$ and
$\nabla U_n(\cdot,y)$ maps $l_{2}(\gamma)$ into $l_{2}^{n}(\gamma)$, then
applying the instance of assumption b) of Proposition \ref{prop:ODE}
which has already been verified for $(\gamma,-\nabla U_n(\cdot,y),n)$ in
the proof of Theorem \ref{thm:main_non_uniform}, we have for any
$x,x^{\prime}\in l_{2}(\gamma)$,
\begin{align}
&\left\langle x-x^{\prime},\nabla U_n(x,y)-\nabla U_n(x^{\prime},y)\right\rangle _{\gamma} \nonumber\\
& =\left\langle \Pi_{n}(x)-\Pi_{n}(x^{\prime}),\nabla U_n(\cdot,y)\circ\Pi_{n}(x)-\nabla U_n(\cdot,y)\circ\Pi_{n}(x^{\prime})\right\rangle _{\gamma}\nonumber \\
 & \geq\lambda\|\Pi_{n}(x)-\Pi_{n}(x^{\prime})\|_{\gamma}^{2}\nonumber \\
 & \to\lambda\|x-x^{\prime}\|_{\gamma}^{2}\quad\mathrm{as}\quad n\to\infty.\label{eq:inner-prod_con}
\end{align}
By (\ref{eq:gradU^N_defn}), (\ref{eq:grad_U_infty_defn}) and assumptions
a) and b) of Condition \ref{assu:hmm2}, we have for any $x\in l_{2}(\gamma)$,
\begin{align}
\|\nabla U_n(x,y)-\partial U(x,y)\|_{\gamma}^{2} & =\gamma^{n}\|\nabla_{n}\tilde{\phi}_n(x,y)-\nabla_{n}\phi_n(x,y)\|^{2}\nonumber \\&+\sum_{k=n+1}^{\infty}\gamma^{k}\|\nabla_{n}\phi_n(x,y)\|^{2}
  \;\to\;0\quad\mathrm{as}\quad n\to\infty.\label{eq:pointwise_con}
\end{align}
Combining (\ref{eq:verify_convexity_inf})\textendash (\ref{eq:pointwise_con})
gives:
\begin{equation}
\left\langle x-x^{\prime},\partial U(x,y)-\partial U(x^{\prime},y)\right\rangle _{\gamma}\geq\lambda\|x-x^{\prime}\|_{\gamma}^{2},\label{eq:dissip_cond_inf}
\end{equation}
which completes the verification of assumption b) of Proposition \ref{prop:ODE}
for the triple $(\gamma,-\partial U(\cdot,y),\infty)$.

The existence of $\xi_{\infty}\in l_{2}(\gamma)$ such that $\partial U(\xi_{\infty},y)=0$,
together with (\ref{eq:dissip_cond_inf}) implies, via the same
arguments as in the proof of Theorem \ref{thm:main_non_uniform},
that equations (\ref{eq:cs_proof}) and (\ref{eq:F^m_bound_proof}) hold
not only for $m\in\mathbb{N}_{0}$ but also for $m=\infty$. Under
assumption a) of Condition \ref{assu:hmm2}, the bound: $\eta_{n}(r,y)\leq\beta_n(y)+\chi r/\gamma$
holds using (\ref{eq:eta_defn}), and plugging in this bound
 completes the proof of \eqref{eq:thm_2_bound}. 

To see that $\|\xi_n-\xi_\infty\|_\gamma\to 0 $ as $n\to\infty$, applying the Cauchy-Schwarz inequality to \eqref{eq:U^n_decay_convexity},
$$
\|x-x^\prime\|_\gamma^2 \leq \frac{1}{\lambda^2}\|\nabla U_n(x,y)-\nabla U_n(x^{\prime},y)\|_\gamma^2,\quad \forall x,x^{\prime}\in l_{2}^{n}(\gamma).
$$
Choosing $x=\xi_n$ and $x^\prime = \Pi_n(\xi_\infty)$,
\begin{align*}
\|\xi_n-\xi_\infty\|_\gamma^2 &= \|\xi_\infty-\Pi_n(\xi_\infty)\|_\gamma^2 + \|\xi_n-\Pi_n(\xi_\infty)\|_\gamma^2\\
&\leq \|\xi_\infty-\Pi_n(\xi_\infty)\|_\gamma^2 + \frac{1}{\lambda^2} \| \nabla U_n(\Pi_n(\xi_\infty),y)\|_\gamma^2.
\end{align*}
Combining with \eqref{eq:pointwise_con},   $\xi_\infty\in l_2(\gamma)$ and parts a) and b) of Condition \ref{assu:hmm2} we find that indeed $\|\xi_n-\xi_\infty\|_\gamma\to 0 $ as $n\to\infty$.

\end{proof}
%


\section{Proofs for section \ref{sec:Discussion-and-application} }\label{sec:discussion_proofs}
\begin{proof}
[Proof of Lemma \ref{lem:AR1}]In the setting described in section
\ref{sec:Discussion-and-application},
\begin{align*}
 & \left\langle x_{n}-x_{n}^{\prime},\nabla_{n}\log f(x_{n-1},x_{n})-\nabla_{n}\log f(x_{n-1}^{\prime},x_{n}^{\prime}) \right\rangle \\
 &\qquad+\left\langle x_{n}-x_{n}^{\prime},\nabla_{n}\log f(x_{n},x_{n+1})-\nabla_{n}\log f(x_{n}^{\prime},x_{n+1}^{\prime})\right\rangle \\
 &\quad =-(x_{n}-x_{n}^{\prime})^{\mathrm{T}}(\Sigma^{-1}+A^{\mathrm{T}}\Sigma^{-1}A)(x_{n}-x_{n}^{\prime})+(x_{n}-x_{n}^{\prime})^{\mathrm{T}}\Sigma^{-1}A(x_{n-1}-x_{n-1}^{\prime})\\
 & \qquad+(x_{n}-x_{n}^{\prime})^{\mathrm{T}}A^{\mathrm{T}}\Sigma^{-1}(x_{n+1}-x_{n+1}),\\
 & \left\langle x_{n}-x_{n}^{\prime},\nabla_{n}\log f(x_{n-1},x_{n})-\nabla_{n}\log f(x_{n-1}^{\prime},x_{n}^{\prime})\right\rangle \\
 &\quad =-(x_{n}-x_{n}^{\prime})^{\mathrm{T}}\Sigma^{-1}(x_{n}-x_{n}^{\prime})+(x_{n}-x_{n}^{\prime})^{\mathrm{T}}\Sigma^{-1}A(x_{n-1}-x_{n-1}^{\prime}),\\
 & \left\langle x_{0}-x_{0}^{\prime},\nabla_{0}\log\mu(x_{0})-\nabla_{0}\log\mu(x_{0}^{\prime})\right\rangle \\
 & \quad =-(x_{0}-x_{0}^{\prime})^{\mathrm{T}}\Sigma_{0}^{-1}(x_{0}-x_{0}^{\prime}),\\
 & \left\langle x_{0}-x_{0}^{\prime},\nabla_{0}\log f(x_{0},x_{1})-\nabla_{0}\log f(x_{0}^{\prime},x_{1}^{\prime})\right\rangle \\
 &\quad =-(x_{0}-x_{0}^{\prime})^{\mathrm{T}}A^{\mathrm{T}}\Sigma^{-1}A(x_{0}-x_{0}^{\prime})+(x_{0}-x_{0}^{\prime})^{\mathrm{T}}A^{\mathrm{T}}\Sigma^{-1}(x_{1}-x_{1}^{\prime}).
\end{align*}
Combining these expressions with (\ref{eq:semi_log_concavity_g}),  (\ref{eq:phi_n_defn})-(\ref{eq:phi_tilde_n_defn}) and applying the following bounds:
\begin{multline}
\inf_{u\neq0}\frac{u^{\mathrm{T}}(\Sigma^{-1}+A^{\mathrm{T}}\Sigma^{-1}A)u}{\|u\|^{2}}\geq\rho_{\min}(\Sigma^{-1})+\rho_{\min}(A^{\mathrm{T}}A)\rho_{\min}(\Sigma^{-1})\\=\rho_{\mathrm{max}}(\Sigma)^{-1}\{1+\rho_{\min}(A^{\mathrm{T}}A)\},
\end{multline}
\begin{align*}
 & \inf_{u\neq0}\frac{u^{\mathrm{T}}\Sigma^{-1}u}{\|u\|^{2}}\wedge\inf_{u\neq0}\frac{u^{\mathrm{T}}(\Sigma_{0}^{-1}+A^{\mathrm{T}}\Sigma^{-1}A)u}{\|u\|^{2}}\\
 & \geq\rho_{\min}(\Sigma^{-1})\wedge\{\rho_{\min}(\Sigma_{0}^{-1})+\rho_{\min}(A^{\mathrm{T}}A)\rho_{\min}(\Sigma^{-1})\}\\
 & =\rho_{\mathrm{max}}(\Sigma)^{-1}\wedge\{\rho_{\mathrm{max}}(\Sigma_{0})^{-1}+\rho_{\min}(A^{\mathrm{T}}A)\rho_{\mathrm{max}}(\Sigma)^{-1}\},
\end{align*}
\[
\sup_{u,v\neq0}\frac{\left|u^{\mathrm{T}}A^{\mathrm{T}}\Sigma^{-1}v\right|}{\|u\|\|v\|}\leq\rho_{\max}(A^{\mathrm{T}}A)^{1/2}\rho_{\mathrm{max}}(\Sigma^{-1})=\rho_{\max}(A^{\mathrm{T}}A)^{1/2}\rho_{\mathrm{min}}(\Sigma)^{-1},
\]
gives the expressions for $\zeta,\tilde{\zeta},\theta$ in the statement
of the lemma.

\end{proof}

\begin{proof}[Proof of Lemma \ref{lem:non_gauss}]
From the Lipschitz assumptions we have:
\begin{align*}
&\|\nabla_{n}\log f(x_{n-1},x_{n})-\nabla_{n}\log f(x_{n-1}^{\prime},x_{n}^{\prime})\| \\
& =\|\nabla\psi\{x_{n}-A(x_{n-1})\}-\nabla\psi\{x_{n}^{\prime}-A(x_{n-1}^{\prime})\}\|\\
 & \leq L_{\nabla\psi}\|x_{n}-x_{n}^{\prime}\|+L_{\nabla\psi}L_{A}\|x_{n-1}-x_{n-1}^{\prime}\|,\\
&\|\nabla_{n}\log f(x_{n},x_{n+1})-\nabla_{n}\log f(x_{n}^{\prime},x_{n+1}^{\prime})\| \\
& \leq\|\nabla A(x_{n})\|_{\mathrm{op}}\|\nabla\psi\{x_{n+1}-A(x_{n})\}-\nabla\psi\{x_{n+1}^{\prime}-A(x_{n}^{\prime})\}\|\\
 & \quad+\|\nabla\psi\{x_{n+1}^{\prime}-A(x_{n}^{\prime})\}\|||\nabla A(x_{n})-\nabla A(x_{n}^{\prime})\|_{\mathrm{op}}\\
 & \leq L_{A}L_{\nabla\psi}\|x_{n+1}-x_{n+1}^{\prime}\|+L_{A}^{2}L_{\nabla\psi}\|x_{n}-x_{n}^{\prime}\|\\
 & \quad+L_{\psi}L_{\nabla A}\|x_{n}-x_{n}^{\prime}\|\\
 & =(L_{A}^{2}L_{\nabla\psi}+L_{\psi}L_{\nabla A})\|x_{n}-x_{n}^{\prime}\|+L_{\nabla\psi}L_{A}\|x_{n+1}-x_{n+1}^{\prime}\|,\\
&\|\nabla\log\mu(x_{0})-\nabla\log\mu(x_{0}^{\prime})\|  =\|\nabla\psi_{0}(x_{0})-\nabla\psi_{0}(x_{0}^{\prime})\|\leq L_{\nabla\psi}\|x_{0}-x_{0}^{\prime}\|.
\end{align*}
The proof is completed by combining these estimates with assumption
b) of the Lemma and (\ref{eq:phi_n_defn})-(\ref{eq:phi_tilde_n_defn}).
\end{proof}

\begin{proof}[Proof of Lemma \ref{lem:poly_moment_growth}]
\begin{multline}
\sum_{n=0}^{\infty}\mathbf{P}\left(\rho^{n}Z_{n} \geq 1\right) = \sum_{n=0}^{\infty}\mathbf{P}\left(|Z_{n}|^s \geq \rho^{-sn}\right)\leq \sum_{n=0}^{\infty} \rho^{sn}\mathbf{E}\left[|Z_{n}|^s \right] \\
\leq \left( \sup_{n\geq0} \frac{\mathbf{E}[|Z_n|^s]}{(n+1)^{p}}\right) \sum_{n=0}^{\infty} (n+1)^p \rho^{sn} <\infty,
\end{multline}
where the first inequality is Markov's inequality. The result follows from the Borel-Cantelli lemma.
\end{proof}

\begin{proof}[Proof of Lemma \ref{lem:alpha_and_beta}]
From (\ref{eq:alpha_and_beta_defns}) and (\ref{eq:phi_n_defn})-(\ref{eq:phi_tilde_n_defn}), there exists a finite constant $c$ such that 
\begin{equation}
\beta_n(y)\leq c +  \|\left.\nabla_{x}\log g(x,y_{n})\right|_{x=0}\|^{2}.\label{eq:beta_bound1}
\end{equation}
The claim of the lemma that $ \sup_{n\geq 0 } \rho^n \beta_{n}(Y)<\infty$ a.s. then follows from the fact that for $a\geq 0 $ and $s\in(0,1]$, the function $a\mapsto a^s $ is subadditive, combined with Lemma \ref{lem:poly_moment_growth}.

 Using the same subadditivity again, and \eqref{eq:beta_bound1},
 $$
 \alpha_{\gamma,n}(y) ^s \leq \sum_{m=0}^{n}\gamma^{s(n-m)}\beta_m(y)^s \leq \frac{c^s}{1-\gamma^s} +  \sum_{m=0}^n \gamma^{s(n-m)}\|\left.\nabla_{x}\log g(x,y_{m})\right|_{x=0}\|^{2s}.
 $$
Combining this bound with the assumption of the lemma and the fact that $\sum_{m=0}^n m^p$ grows no faster than $n^{p+1}$ as $n\to
\infty$, there exists a  finite constant $c^\prime(p)$ such that:
\begin{align*}
 \mathbf{E}[ \alpha_{\gamma,n}(Y) ^s] &\leq \frac{c^s}{1-\gamma^s} + \sum_{m=0}^n  \mathbf{E}\left[\|\left.\nabla_{x}\log g(x,Y_{m})\right|_{x=0}\|^{2s} \right]\\
 &  \leq \frac{c^s}{1-\gamma^s} +c^{\prime}(p) n^{p+1}.
  \end{align*}
  The claim of the lemma that $\sup_{n\geq 0 } \rho^n \alpha_{\gamma,n}(Y)<\infty$ a.s. then follows by applying Lemma \ref{lem:poly_moment_growth}.  The proof is completed via bounds:
 $$
 \rho^n \alpha_{\gamma,n}(Y) =\tilde{\rho}^n \left(\frac{\rho}{\tilde{\rho}}\right)^n \alpha_{\gamma,n}(Y)\leq  \tilde{\rho}^n \sup_{m\geq 0} \left(\frac{\rho}{\tilde{\rho}}\right)^m  \alpha_{\gamma,m}(Y)
 $$
 and
 $$
 \sum_{k=n}^{\infty} \rho^k \beta_k(Y) =  \sum_{k=n}^{\infty} \tilde{\rho}^k \left(\frac{\rho}{\tilde{\rho}}\right)^k \beta_k(Y) \leq \frac{\tilde{\rho}^n}{1-\tilde{\rho}}  \sup_{m\geq 0}\left(\frac{\rho}{\tilde{\rho}}\right)^m \beta_m(Y). 
 $$
 
\end{proof}

\end{appendix}

\bibliographystyle{plainnat} 
\bibliography{filtering}       


\end{document}